\documentclass[10pt,reqno]{amsart}
\usepackage{amsmath, amssymb, amsthm}
\usepackage[utf8]{inputenc}
\usepackage{mathrsfs}
\usepackage{tikz}
\usepackage{algorithm}
\usepackage{algorithmicx}
\usepackage{algpseudocode}
\usepackage{alltt}
\usepackage{float}
\usepackage{url}

\DeclareMathOperator\Dyn{Dyn}

\newcommand{\minus}{\mathchoice{\!-\!}{\!-\!}{\raise0.7pt\hbox{$\scriptscriptstyle-$}\scriptstyle}{-}}
\newcommand{\plus}{\mathchoice{\!+\!}{\!+\!}{\raise0.7pt\hbox{$\scriptscriptstyle+$}\scriptstyle}{+}}
\newcommand{\newintegeri}[2]{
	\expandafter\def\csname #1\endcsname{#2}
	\expandafter\def\csname #1o\endcsname{{#2\minus1}}
	\expandafter\def\csname #1t\endcsname{{#2\minus2}}
	\expandafter\def\csname #1p\endcsname{{#2\plus1}}
	\expandafter\def\csname #1pp\endcsname{{#2\plus2}}}
\newcommand{\newintegerii}[1]{\newintegeri{#1#1}{#1}}
\newintegeri{indi}{p}
\newintegeri{indii}{q}
\newintegeri{indiii}{r}
\newintegeri{indiv}{s}
\newintegeri{indv}{t}
\newintegeri{indvi}{t'}
\newintegeri{ll}{\ell}
\newintegerii{d}
\newintegerii{e}
\newintegerii{h}
\newintegerii{i}
\newintegerii{j}
\newintegerii{k}
\newintegerii{m}
\newintegerii{n}
\newintegerii{p}
\newintegerii{q}
\newintegerii{r}
\renewcommand{\aa}[2]{a_{#1,#2}}
\newcommand{\aainv}[2]{a_{#1,#2}\inv}
\newcommand{\ABKL}[1]{\Sigma_{#1}^{\scriptstyle*}}
\newcommand{\ABKLP}[1]{\Sigma_{#1}\bidual}
\newcommand{\ABP}[1]{\Sigma_{#1}}
\newcommand{\ABPP}[1]{\Sigma_{#1}^{\raise0.8pt\hbox{$\scriptscriptstyle+$}}}
\newcommand{\app}{\triangleleft}

\newcommand{\BB}[1]{B_{#1}}
\newcommand{\bidual}{^{\hspace{-0.05em}\raise0.6pt\hbox{$\scriptscriptstyle+$}\hspace{-0.05em}*}}
\newcommand{\BKL}[1]{B_{#1}\bidual}
\newcommand{\BP}[1]{B_{#1}\plusexp}
\newcommand{\br}{\beta}
\newcommand{\brbr}{\gamma}

\newcommand{\card}[1]{\mathrm{card}\left(#1\right)}

\newcommand{\ew}{\varepsilon}

\renewcommand{\geq}{\geqslant}
\newcommand{\geo}[3]{g(#1,#2;#3)}

\newcommand{\hash}{\text{hash}}
\newcommand\HH{\hspace{1em}}

\newcommand{\ie}{\emph{i.e.}}
\newcommand{\inv}{^{\minus\hspace{-0.1em}1}}

\renewcommand{\leq}{\leqslant}
\newcommand{\len}[2]{\vert #2\vert_#1}

\newcommand{\napp}{\ntriangleleft}
\newcommand\NN{\mathbb{N}}

\newcommand{\one}[1]{\mathbb{1}}
\newcommand{\plusminus}{\mathchoice{\pm}{\pm}{\raise0.7pt\hbox{$\scriptscriptstyle\pm$}\scriptstyle}{\pm}}
\newcommand{\plusexp}{^{\raise0.8pt\hbox{$\hspace{-0.05em}\scriptscriptstyle+$}}}

\newcommand{\red}{\text{red}}
\newcommand\rem{\text{rem}}

\newcommand\resp{\emph{resp.~}}
\newcommand\RR[1]{\mathrm{R}_#1}

\newcommand{\sig}[1]{\sigma_{\!#1}} 
\newcommand{\siginv}[1]{\sigma_{\!#1}^{\hspace{-0.05em}\raise0.8pt\hbox{$\scriptscriptstyle-$}\hspace{-0.1em}1}}

\newcommand{\sigpm}[1]{\sigma_{\!#1}^{\hspace{-0.05em}\raise0.8pt\hbox{$\scriptscriptstyle\pm$}\hspace{-0.1em}1}}

\newcommand{\sph}[3]{s(#1,#2;#3)}
\newcommand{\SSS}[1]{S_{#1}}
\newcommand{\Sym}[1]{\mathfrak{S}_#1}

\newcommand{\TT}[1]{T_{#1}}

\newcommand{\uu}{u}

\newcommand{\vv}{v}

\newcommand{\ww}{w}
\newcommand{\WW}{W}

\newcommand{\xx}{x}

\newcommand{\yy}{y}

\newcommand{\ZZ}{\mathbb{Z}}

\theoremstyle{definition}

\newtheorem{defi}{Definition}[section]
\newtheorem{nota}[defi]{Notation}
\newtheorem*{rema}{Remark}
\newtheorem{exam}[defi]{Example}

\theoremstyle{plain}

\newtheorem{lemm}[defi]{Lemma}
\newtheorem{prop}[defi]{Proposition}

\newtheorem{coro}[defi]{Corollary}
\newtheorem{conj}[defi]{Conjecture}

\title{Experiments on growth series of braid groups}

\author[1]{Jean Fromentin}

\address{Univ. Littoral Côte d’Opale, UR 2597, LMPA, Laboratoire de Mathématiques Pures et Appliquées Joseph Liouville, F-62100 Calais, France}
\email{fromentin@math.cnrs.fr}


\dedicatory{In memory of Patrick Dehornoy, a great mentor.}

\begin{document}

\begin{abstract}
We introduce an algorithmic framework to investigate spherical and geodesic growth series of braid groups relatively to the Artin's or Birman--Ko--Lee's generators. We  present our experimentations in the case of three and four strands and conjecture rational expressions for the spherical growth series with respect to the Birman--Ko--Lee's generators.  
\end{abstract}

\maketitle

\section{Introduction}

Originally introduced as the group of isotopy classes of $n$-strands geometric braids, %~\cite{Artin25}, 
the braid group $\BB\nn$ admits many finite presentations by generators and relations.
From each finite semigroup generating set $S$ of $\BB\nn$ we can define at least two growth series.
The spherical growth series counts elements of $\BB\nn$ by their distance from the identity in the Cayley 
graph $\text{Cay}(\BB\nn,S)$ of~$\BB\nn$ with respect to~$S$.
The geodesic growth series counts geodesic paths starting from the identity by length in $\text{Cay}(\BB\nn,S)$.

In case of Artin's generators $\Sigma_n=\{\sigpm1,\ldots,\sigpm\nno\}$ of $\BB\nn$
 the only known significant results are for $n\leq3$.
L. Sabalka determines \cite{Sabalka} both the spherical and geodesic growth series of $\BB3$.
To this end, he constructs an explicit deterministic finite automaton recognizing the language of
geodesic $\Sigma_3$-words. 
In particular he obtains the rationality of both series.  
Similar results were obtained by J. Mairesse and F. Mathéus in case of Artin--Tits groups of dihedral 
type~\cite{Mairesse}.
In page~57 of her PhD thesis~\cite{Albenque-thesis},  M. Albenque gives the first $13$ terms of the spherical series of~$\BB4$ relatively to $\Sigma_4$.

Here we introduce a new algorithmic framework to compute the first terms of the spherical and geodesic growth series of $\BB\nn$ relatively to both  Artin's or Birman--Ko--Lee's generators. 
Experimentations allow us to conjecture rational expressions for the spherical growth series of $\BB3$ and $\BB4$ and geodesic growth series of $\BB3$ relatively to the Birman--Ko--Lee's generators.
We also obtain the first $26$ terms of the spherical and geodesic growth series of $\BB4$ with respect to~$\Sigma_4$ but this is not enough to formulate any conjecture in this case.
Experiments presented in this paper were carried out using the \texttt{CALCULCO} computing platform~\cite{Calculco}.

The paper is organized as follows.
Section~\ref{S:Context} recalls basic definitions and presents already known result on the subject.
In section~\ref{S:Counting} we describe a first algorithm to explore spherical and geodesic combinatorics of braids relatively to Artin's or Birman--Ko--Lee generators.
Section~\ref{S:Template} is devoted to the notion of braid template which allows us to parallelize the algorithms obtained in the previous section. 
In section~\ref{S:Reduced} we show how to reduce the exploration space by introducing reduced braid templates.
Experimentation results are detailed in the last section.

\section{Context}

\label{S:Context}
\subsection{Growth series}

Let $S$ be a finite generating set of a semigroup $M$. 
We denote by $S^\ast$ the set of all words on the alphabet $S$, which are called $S$-words.
The empty word is denoted by $\ew$.
For every $S$-word $u$, we denote by $|u|$ its length and  by $\overline{u}$ the element of $M$ it represents.
We say that two $S$-words $\uu$ and $\vv$ are equivalent, denoted $\uu\equiv\vv$, is they represent the same element in $M$.

\begin{defi}
The \emph{S-length} of an element $x\in M$, denoted $\len{S}x$, is the length of a shortest $S$-word representing $x$. 
An $S$-word $\uu$ satisfying $|\uu|=\len{S}{\overline{\uu}}$ is \emph{geodesic}.
\end{defi}

The $S$-length of an element $x\in M$ corresponds to the distance between $x$ and the identity in the Cayley graph of $M$ with respect to the finite generating set $S$.

\begin{defi}
For any $\ell\in\NN$, we denote by $\geo{M}S\ell$ the number of geodesic $S$-words of length $\ell$.
The \emph{geodesic} growth series of $M$ with respect to $S$ is 
\[
\mathcal{G}(M,S)=\sum_{\ell\in\NN} g(M,S;\ell)\, t^\ell.
\]
\end{defi}

If the language of geodesic $S$-words is regular then the series $\mathcal{G}(M,S)$ is rational.

\begin{defi}
For any $\ell\in\NN$, we denote by $\sph{M}S\ell$ the number of elements in~$M$ of length $\ell$.
The \emph{spherical} growth series of $M$ with respect to $S$ is 
\[
 \mathcal{S}(M,S)=\sum_{x \in M} t^{\len{S}x}=\sum_{\ell\in \NN} \sph{M}S\ell t^\ell.
\]
\end{defi}

If there exists a regular language composed of geodesic $S$-words in bijection with~$M$ then the series $\mathcal{S}(M,S)$ is rational.

\subsection{Artin's braid presentation}
The first presentation of the braid group $\BB\nn$ was given by E. Artin in \cite{Artin47} :

\begin{equation}
\label{E:PBn}
\BB\nn\simeq\left<\sig1,...,\sig\nno\left| \begin{array}{cl} \sig\ii\sig\jj=\sig\jj\,\sig\ii 
& \text{for $|\ii-\jj|\geq 2$} \\ \sig\ii\,\sig\jj\,\sig\ii=\sig\jj\,\sig\ii\,\sig\jj
&\text{for $|\ii-\jj|=1$}\end{array}\right. \right>.
\end{equation}

\begin{defi}
 For all $n\geq 2$, we denote by $\ABPP\nn$  the set $\{\sig1,\ldots,\sig\nno\}$ and by~$\ABP\nn$ the set $\ABPP\nn\sqcup\left(\ABPP\nn\right)\inv$. 
\end{defi}

Artin's presentation of $\BB\nn$ implies that $\ABPP\nn$ is a set of group generators of~$\BB\nn$.
However the braid $\siginv1$ cannot be represented by any $\ABPP\nn$-word.
For our purpose, it is fundamental to view a monoid (or a group) as a quotient of a finitely generated free monoid.
As a monoid, the braid group $\BB\nn$ is presented by generators~$\ABP\nn$ and the relations of \eqref{E:PBn} plus relations 
\begin{equation}
\label{E:Pssinv}
 \sig\ii\,\siginv\ii=\siginv\ii\,\sig\ii=\ew\quad \text{for all $1\leq \ii \leq \nno$.}
\end{equation}

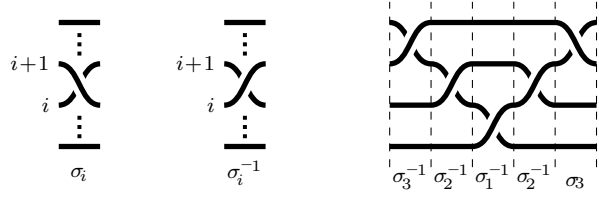
\begin{figure}[h!]
\begin{center}
    \begin{tikzpicture}[x=0.055cm,y=0.055cm]
      \draw[line width=2](0,0) -- (10,0);
      \draw[line width=2](0,10) .. controls (5,10) and (5,20) .. (10,20);
      \draw[line width=2](0,20) .. controls (5,20) and (5,10) .. (10,10);
   \draw[line width=6,color=white](0,20) .. controls (5,20) and (5,10) .. (10,10);
         \draw[line width=2](0,20) .. controls (5,20) and (5,10) .. (10,10);
      \draw[line width=2](0,30) -- (10,30);
      \draw[line width=1.5,dotted](5,2) -- (5,8); 
      \draw[line width=1.5,dotted](5,22) -- (5,28); 
      \draw(0,10) node[left]{\footnotesize$\ii$};
      \draw(0,20) node[left]{\footnotesize$\iip$};
      \draw(5,-3) node[below]{\footnotesize$\sig\ii$};
      \begin{scope}[shift={(40,0)}]
        \draw[line width=2](0,0) -- (10,0);
        \draw[line width=2](0,10) .. controls (5,10) and (5,20) .. (10,20);
        \draw[line width=2](0,20) .. controls (5,20) and (5,10) .. (10,10);
        \draw[line width=6,color=white](0,10) .. controls (5,10) and (5,20) .. (10,20);
        \draw[line width=2](0,10) .. controls (5,10) and (5,20) .. (10,20);
        \draw[line width=2](0,30) -- (10,30);
        \draw[line width=1.5,dotted](5,2) -- (5,8); 
        \draw[line width=1.5,dotted](5,22) -- (5,28); 
        \draw(0,10) node[left]{\footnotesize$\ii$};
        \draw(0,20) node[left]{\footnotesize$\iip$};
        \draw(5,-1) node[below]{\footnotesize$\siginv\ii$};
      \end{scope}
      \begin{scope}[shift={(80,0)}]
        \draw[line width=2](0,30) .. controls (5,30) and (5,20) .. (10,20) .. controls (15,20) and (15,10) .. (20,10) .. controls (25,10) and (25,0) .. (30,0) -- (50,0);
        \draw[line width=6,color=white](0,10) -- (10,10) .. controls (15,10) and (15,20) .. (20,20) -- (30,20) .. controls (35,20) and (35,10) .. (40,10) -- (50,10);
        \draw[line width=2](0,10) -- (10,10) .. controls (15,10) and (15,20) .. (20,20) -- (30,20) .. controls (35,20) and (35,10) .. (40,10) -- (50,10);
        \draw[line width=6,color=white] (0,0) -- (20,0) .. controls (25,0) and (25,10) .. (30,10) .. controls (35,10) and (35,20) .. (40,20) .. controls (45,20) and (45,30) .. (50,30);
        \draw[line width=2] (0,0) -- (20,0) .. controls (25,0) and (25,10) .. (30,10) .. controls (35,10) and (35,20) .. (40,20) .. controls (45,20) and (45,30) .. (50,30);
        \draw[line width=6,color=white] (0,20) .. controls (5,20) and (5,30) .. (10,30) -- (40,30) .. controls (45,30) and (45,20) .. (50,20);
        \draw[line width=2] (0,20) .. controls (5,20) and (5,30) .. (10,30) -- (40,30) .. controls (45,30) and (45,20) .. (50,20);
        \draw[dashed](0,-5) -- (0,35);
        \draw[dashed](10,-5) -- (10,35);
        \draw[dashed](20,-5) -- (20,35);
        \draw[dashed](30,-5) -- (30,35);
        \draw[dashed](40,-5) -- (40,35);
        \draw[dashed](50,-5) -- (50,35);
        \draw(5,-2) node[below]{\footnotesize$\siginv3$};
        \draw(15,-2) node[below]{\footnotesize$\siginv2$};
        \draw(25,-2) node[below]{\footnotesize$\siginv1$};
        \draw(35,-2) node[below]{\footnotesize$\siginv2$};
        \draw(45,-5) node[below]{\footnotesize$\sig3$};
      \end{scope}
      \end{tikzpicture}
      \caption{Geometric interpretation of Artin's generators and representation of a $4$-strands braid as a $\ABP4$-word.}
      \end{center}
      \label{F:mots} 
\end{figure}

In \cite{Sabalka}, L. Sabalka constructed an explicit deterministic finite state automaton recognizing the language of geodesic $\ABP3$-words. 
He obtained the following rational value for the geodesic growth series of $\BB3$ relatively to the Artin's generators $\ABP3$ :
\begin{equation}
\mathcal{G}(\BB3,\ABP3)=\frac{t^4+3t^3+t+1}{(t^2+2t-1)(t^2+t-1)}.
\end{equation}
Moreover, using the finite state automaton recognizing the language of short-lex normal form of $\BB3$ \cite{EpsteinHoltRees} he obtains :
\begin{equation}
\mathcal{S}(\BB3,\ABP3)=\frac{(t+1)(2t^3-t^2+t-1)}{(t-1)(2t-1)(t^2+t-1)}.
\end{equation}

The positive braid monoid $\BP\nn$ is the submonoid of $\BB\nn$ generated by $\ABPP\nn$.
Since every $\ABPP\nn$-word is geodesic, the geodesic growth series $\mathcal{G}(\BP\nn,\ABPP\nn)$ is irrelevant.
An explicit rational formula for the spherical growth series $\mathcal{S}(\BP\nn,\ABPP\nn)$ was obtained by A. Bronfman in \cite{Bronfman} and later by M. Albenque in \cite{Albenque-bpn}. 
These results were extended to positive braid monoids of type \texttt{B} and \texttt{D} in \cite{AlbenqueNadeau} and for each Artin--Tits monoids of spherical type in~\cite{FloresMeneses-combi}.

\subsection{Dual's braid presentation}

In \cite{BKL}, J. Birman, K.\,H.~Ko and S.\,J.~Lee introduced a new generator family of $\BB\nn$, called Birman-Ko-Lee's or dual generators.

\begin{defi}
\label{D:apq}
 For $1\leq \indi< \indii \leq n$ we  define $\aa\indi\indii$ to be the braid
 \begin{equation}
 \label{E:apq}
\aa\indi\indii=\sig\indi\ldots\sig\indiit\ \sig\indiio\ \siginv\indiit\ldots\siginv\indi.
 \end{equation}
For all $n\geq 2$, we put $\ABKLP\nn=\{\aa\indi\indii\,|\,1\leq\indi<\indii\leq n\}$ and $\ABKL\nn=\ABKLP\nn\sqcup\left(\ABKLP\nn\right)\inv$.
\end{defi}

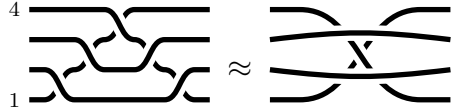
\begin{figure}[h!]
  \begin{center}
    \begin{tikzpicture}[x=0.040cm,y=0.040cm]
      \draw[line width=2] (-5,0) -- (0,0) .. controls (5,0) and (5,10) .. (10,10) .. controls (15,10) and (15,20) .. (20,20) .. controls (25,20) and (25,30) .. (30,30) -- (55,30); 
      \draw[line width=6,color=white] (-5,30) -- (20,30) .. controls (25,30) and (25,20) .. (30,20) .. controls (35,20) and (35,10) .. (40,10) .. controls (45,10) and (45,0) .. (50,0) -- (55,0); 
      \draw[line width=2] (-5,30) -- (20,30) .. controls (25,30) and (25,20) .. (30,20) .. controls (35,20) and (35,10) .. (40,10) .. controls (45,10) and (45,0) .. (50,0) -- (55,0); 
      \draw[line width=6,color=white] (-5,10) -- (0,10) .. controls (5,10) and (5,0) .. (10,0) -- (40,0) .. controls (45,0) and (45,10) .. (50,10) -- (55,10);
      \draw[line width=2] (-5,10) -- (0,10) .. controls (5,10) and (5,0) .. (10,0) -- (40,0) .. controls (45,0) and (45,10) .. (50,10) -- (55,10);
      \draw[line width=6,color=white] (-5,20) -- (10,20) .. controls (15,20) and (15,10) .. (20,10) -- (30,10) .. controls (35,10) and (35,20) .. (40,20) -- (55,20);
      \draw[line width=2] (-5,20) -- (10,20) .. controls (15,20) and (15,10) .. (20,10) -- (30,10) .. controls (35,10) and (35,20) .. (40,20) -- (55,20);
      \draw(-5,0) node[left]{\footnotesize$1$};
      \draw(-5,30) node[left]{\footnotesize$4$};
      \draw(65,10) node{\Large$\approx$};
      \begin{scope}[shift={(75,0)}]
        \draw[line width=2](0,0) -- (10,0) .. controls (30,0) and (30,30) .. (50,30) -- (60,30); 
        \draw[line width=6,color=white](0,30) -- (10,30) .. controls (30,30) and (30,0) .. (50,0) -- (60,0); 
        \draw[line width=2](0,30) -- (10,30) .. controls (30,30) and (30,0) .. (50,0) -- (60,0); 
        \draw[line width=6,color=white](0,10) .. controls (25,7) and (35,7) .. (60,10);
        \draw[line width=2](0,10) .. controls (25,7) and (35,7) .. (60,10);
        \draw[line width=6,color=white](0,20) .. controls (25,23) and (35,23) ..  (60,20);
        \draw[line width=2](0,20)  .. controls (25,23) and (35,23) ..  (60,20);
      \end{scope}
    \end{tikzpicture}
  \end{center}
  \caption{The letter $\aa14$ codes for the braid in which strands $1$ and $4$ cross under strands $2$ and $3$.}
  \label{F:apq}
\end{figure}

We write $[\indi,\indii]$ for the interval $\{\indi,\ldots , \indii\}$ of $\NN$, and we say that $[\indi, \indii]$ is nested in~$[\indiii, \indiv]$ if we have $\indiii < \indi < \indii < \indiv$.

\begin{lemm}\cite{BKL}
In terms of $\ABKLP\nn$, the group $\BB\nn$ is presented by the relations
\begin{align}
\label{E:DualCommutativeRelation}
\aa\indi\indii\aa\indiii\indiv&= \aa\indiii\indiv\aa\indi\indii \text{\quad for $[\indi, \indii]$ and $[\indiii, \indiv]$ disjoint or nested},\\
\label{E:DualNonCommutativeRelation}
\aa\indi\indii\aa\indii\indiii&= \aa\indii\indiii\aa\indi\indiii =\aa\indi\indiii\aa\indi\indii \text{\quad for $1 \le \indi<\indii<\indiii\le\nn$}.
\end{align}
\end{lemm}
Note that the definition of $\aa\indi\indii$ given here is not exactly that of \cite{BKL} but it is coherent with previous papers of the author. 

As for Artin's generators, the braid group $\BB\nn$ admits a monoid presentation with generators $\ABKL\nn$, relations \eqref{E:DualCommutativeRelation} and \eqref{E:DualNonCommutativeRelation} together with
\begin{equation}
\label{E:DualInverses}
 \aa\indi\indii\,\aainv\indi\indii=\aainv\indi\indii\,\aa\indi\indii=\ew\quad \text{for all $1\leq \indi<\indii\leq \nn$.}
\end{equation}

Except in the case $n=2$, which is trivial, there are no results in the literature on the growth series of~$\BB\nn$ with respect to $\ABKL\nn$.

The Birman--Ko--Lee monoid~$\BKL\nn$, also called dual braid monoid in \cite{Bessis} is the submonoid of~$\BB\nn$ generated by $\ABKLP\nn$.
The term \emph{dual} was used by D. Bessis since the Garside structure of $\BP\nn$ and $\BKL\nn$ share symmetric combinatorial values.
In~\cite{AlbenqueNadeau}, M. Albenque and P. Nadeau give a rational expression for the spherical growth series~$\mathcal{S}(\BKL\nn,\ABKLP\nn)$; they also treat the case of dual braid monoids of type~\textbf{B}.

\subsection{Some words about Garside presentations}

The two monoids $\BP\nn$ and $\BKL\nn$ equip the braid group $\BB\nn$ with two Garside structures : the classical one~\cite{Garside} and the dual one~\cite{BKL,Bessis}. 
The reader can consult~\cite{DehornoyParis} and \cite{DehornoyAl} for a general introduction to Garside theory.
Here it is sufficient to  know that each Garside structure provides \emph{simple} elements which generate the corresponding Garside monoid.
Let us denote by $C_n$ and $D_n$ the simple elements of the Garisde monoid $\BP\nn$ and $\BKL\nn$ respectively.

In \cite{Dehornoy-combinatorics}, P. Dehornoy starts the study of the spherical combinatorics of~$\BP\nn$ relatively to $C_n$. In particular he formulates a divisibility conjecture which has been proven by F. Hivert, J.-C.~Novelli and J.-Y.~Thibon in \cite{HivertNovelliThibon}. 
A similar result was obtained for braid monoids of type \textbf{B} in \cite{FoissyFromentin}.
The spherical combinatorics of $\BKL\nn$ relatively to $D_\nn$ was also considered by P. Biane and P. Dehornoy in~\cite{Biane}: they reduce the computation of $\sph{\BKL\nn}{D_\nn}2$ to that of free cumulants for a product of independent variables.

R. Charney establishes in \cite{Charney} that the spherical growth series of Artin--Tits groups of spherical type with respect to their standard simple elements are rationals. In particular she obtains the rationality of $\mathcal{S}(\BB\nn,C_n)$. This result was generalized for all Garside groups by P. Dehornoy in \cite{DehornoyGarside}.
This implies in particular the rationality of $\mathcal{S}(\BB\nn,D_n)$.

\section{Counting braids}

\label{S:Counting}

We fix an integer $n\geq 2$ and $\SSS\nn$ denotes either $\ABP\nn$ (Artin's generators of $\BB\nn$) or~$\ABP\nn^\ast$ (dual generators of $\BB\nn$).

\begin{defi}
For $\nn\geq 2$ and $\ell\in\NN$ we denote by $\BB\nn(\SSS\nn,\ell)$ the set of braids of~$\BB\nn$ whose $S_\nn$-length is $\ell$.
\end{defi}

Since the equality $s(\BB\nn,S_\nn;\ell)=\card{\BB\nn(\SSS\nn,\ell)}$ holds, we compute $s(\BB\nn,\SSS\nn;\ell)$ by constructing the set $\BB\nn(\SSS\nn,\ell)$.
Each braid of $\BB\nn$ with $\SSS\nn$-length  $\ell$ is the product of a braid of $\SSS\nn$-length~$\llo$ and a generator $x\in \SSS\nn$.
In particular we have 
\begin{equation}
\label{E:B:sub}
\BB\nn(\SSS\nn,\ll)\subseteq \{ \br\cdot x\  \text{for}\ (\br,x)\in\BB\nn(\SSS\nn,\llo)\times\SSS\nn\},
\end{equation}
and so we can construct $\BB\nn(\SSS\nn,\ell)$ by induction on $\ell\geq 1$.

\subsection{Representative sets}

From an algorithmic point of view, a braid is naturally represented by a word.
We extend this notion to any subset of $\BB\nn(\SSS\nn,\ell)$.

\begin{defi}
\label{D:Rep}
We say that a set $\WW$ of $S_\nn$-words represents a subset $X$ of $\BB\nn$ whenever $W$ is a set of unique geodesic representatives for $X$.
\end{defi}

\begin{exam}
For all $\nn\geq 2$, the set $\{\ew\}$ represents $\BB\nn(\SSS\nn,0)$.
Since relations~\eqref{E:PBn},~\eqref{E:Pssinv}, together with relations \eqref{E:DualCommutativeRelation}-\eqref{E:DualInverses} of Artin and dual semigroup presentation of $\BB\nn$ preserve parity of word length we have the following property:
\begin{equation}
\label{E:mod2}
\text{two $S_\nn$-words $\uu$ and $\vv$ are equivalent only if $|u|\equiv|v| \text{ mod } 2$.}
\end{equation}
In particular any $S_\nn$-word of length $\leq 1$ is geodesic. 
As two different letters of~$\SSS\nn$ represent different braids of $\BB\nn$ the set $\SSS\nn$ represents $\BB\nn(\SSS\nn,1)$.
\end{exam}

The previous example gives a representative set of~$\BB\nn(\SSS\nn,\ell)$ for $\ell\leq1$.
We now tackle the construction of a representative set $\WW_\ll$ of $\BB\nn(\SSS\nn,\ell)$ for $\ll\geq2$.
Using an inductive argument we can assume we already have obtained a set~$\WW_\llo$ representing~$\BB\nn(\SSS\nn,\llo)$ and then consider the set 
\begin{equation}
\label{E:X}
\WW'=\{\ww\xx\ \text{for}\ (\ww,\xx)\in \WW_\llo\times \SSS\nn\}.
\end{equation}
A first step to obtain $\WW_\ell$ consists in removing all non-geodesic words from~$\WW'$.
For this we have to test if a given word of~$\WW'$ is geodesic or not.
A naive general solution consists in testing if a word $\uu\in\WW'$ is equivalent to a $\SSS\nn$-word of length at most $\llo$.
However, as words of $\WW'$ are obtained by appending a letter to a geodesic word, we can restrict the search space:

\begin{lemm}
\label{L:Geo}
For $\ell\geq 2$, let $\uu$ be a geodesic $S_\nn$-word of length $\ell-1$ and $x$ a letter of $S_\nn$. 
If the $S_\nn$-word $\vv=\uu\xx$ is not geodesic then there exists a geodesic $S_\nn$-word~$\ww$ of length~$\ell-2$ which is equivalent to $\vv$. \end{lemm}

\begin{proof}
Assume~$\vv$ is not geodesic. 
There exists a $S_\nn$-geodesic word $\ww$ equivalent to~$\vv$ and satisfying $|\ww|<|\vv|$.
By \eqref{E:mod2} we must have $|\ww|\leq|\vv|-2=\ell-2$.
From the equality $\vv=\uu\xx$ we obtain~$\uu\equiv\vv\xx\inv$ and so $\uu\equiv\ww\xx\inv$.
Since $\uu$ is geodesic we must have $|\ww\xx\inv|\geq\ell-1$, implying $|\ww|\geq\ell-2$ and so $|\ww|=\ell-2$.
\end{proof}

\subsection{Geodesic words}
 For all $\ell\in\NN$ the number $g(\BB\nn,\SSS\nn;\ell)$ can be obtained at no cost during the construction of a representative set of $\BB\nn(\SSS\nn,\ell)$.
\begin{defi}
For a braid $\br\in\BB\nn$ we denote by $\omega_{\SSS\nn}(\br)$ the number of geodesic $\SSS\nn$-words representing $\br$. 
\end{defi}

\begin{prop}
\label{P:NGeo}
For $\br\in\BB\nn$ a braid with $\ell=|\br|_{\SSS\nn}\geq 1$, we have
\[
 \omega_{\SSS\nn}(\br)=\sum_{\substack{x \in \SSS\nn\\ |\br\xx\inv|_{\SSS\nn}=\llo}} \omega_{\SSS\nn}(\br\xx\inv).
\]
\end{prop}

\begin{proof}
Let $\WW$ be the set of geodesic $\SSS\nn$-words representing $\br$. 
The cardinality of~$\WW$ is then $\omega_{\SSS\nn}(\br)$.
For all $x\in\SSS\nn$ we denote by $\WW_\xx$ the words of $\WW$ ending with $\xx$.
Since all words of $\WW$ have length $\ell\geq 1$ we must have
\[
\WW=\bigsqcup_{x\in\SSS\nn} \WW_\xx.
\]
Let us fix an element $\yy\in\SSS\nn$.
By construction, any word of $\WW_\yy$ has length $\llo$, represents the braid $\br\yy\inv$ and is geodesic.
Hence $\WW_\yy$ is not empty if and only if the $\SSS\nn$-length of $\br\yy\inv$ is $\llo$, which gives 
\[ \omega_{\SSS\nn}(\br)=\card{\WW}=\sum_{x\in\SSS\nn}\card{\WW_\xx}=\hspace{-1.5em}\sum_{\substack{x \in \SSS\nn\\ |\br\xx\inv|_{\SSS\nn}=\llo}} \hspace{-1.5em}\card{\WW_\xx}.
\]
Assume $\br\yy\inv$ has $\SSS\nn$-length $\llo$.
Since for any geodesic $\SSS\nn$-word $\vv$ representing~$\br\yy\inv$, the word $\vv\yy$  is a geodesic representative of $\br$, the braid $\br\yy\inv$ has exactly~$\omega_{\SSS\nn}(\br\yy\inv)$ geodesic representatives in~$\WW_\yy$.
Therefore $\card{\WW_\yy}$ is $\omega_{\SSS\nn}(\br\yy\inv)$ and the result follows.
\end{proof}

\subsection{A first algorithm}

We can now give a first algorithm returning a representative set $\WW_\ell$ of~$\BB\nn(\SSS\nn,\ell)$ for $\ell\geq 2$.
In order to determine $g(\BB\nn,\SSS\nn;\ell)$ we also compute the value of $\omega_{\SSS\nn}$ for all words in $\WW_\ell$. 

In order to construct by induction a representative set $W$, we must test if a given word $u$ is equivalent to a word occuring in $W$ :

\begin{defi}
For a set $\WW$ of $\SSS\nn$-words we say that a $\SSS\nn$-word $\uu$ \emph{appears} in~$\WW$, denoted by $\uu\app \WW$, if $\uu$ is equivalent to a word $\vv$ of $\WW$.
\end{defi}

In an algorithmic context a $\SSS\nn$-word is represented as an array of integers plus another integer $\omega$ which eventually correspond to $\omega_{\SSS\nn}(\overline\uu)$.
Whenever two variables $\texttt{u}$ and $\texttt{v}$ stand for the $\SSS\nn$-words $\uu$ and $\vv$ we use:

-- $\texttt{u}\cdot\omega$ to design the integer $\omega$ associated to the word $\uu$;

-- $\texttt{u}\,\texttt{v}$ to design the product $\uu\vv$.

\begin{algorithm}[H]
\caption{-- \small \textsc{RepSet} : For $\ell \geq 2$, returns a set $\texttt{\WW}_\ll$ representing $\BB\nn(\SSS\nn,\ell)$ from two sets $\texttt{\WW}_\llo$ and $\texttt{\WW}_\llt$ representing $\BB\nn(\SSS\nn,\llo)$ and $\BB\nn(\SSS\nn,\llt)$ respectively. For each word $\uu\in\texttt{\WW}_\llo$, the value of $\texttt{u}\cdot\omega$ is assumed to be $\omega_{\SSS\nn}(\overline{\uu})$.}
\label{A:Rep}
\small
\begin{algorithmic}[1]
\Function{RepSet}{$\texttt{\WW}_{\llo}$,$\texttt{\WW}_\llt$}
\State $\texttt{\WW}_\ll\gets \emptyset$
\For{$\texttt{\xx}\in \SSS\nn$}
	\For{$\texttt{\uu}\in \texttt{\WW}_\llo$}
		\State $\texttt{\vv}\gets \texttt{\uu}\,\texttt{\xx}$
		\If{$\texttt{\vv}\napp \texttt{\WW}_\llt$}
            \If{$\texttt{\vv}\napp\texttt{\WW}_\ll$} \Comment{a new braid $\overline\vv$ of $\SSS\nn$-length $\ll$ is found}
                \State $\texttt{\WW}_\ll\gets \texttt{\WW}_\ll\sqcup \{\texttt{\vv}\}$
                \State $\texttt{\vv}\cdot\omega\gets\texttt{\uu}\cdot\omega$
            \Else \Comment{$\vv$ is another geodesic word representing $\overline\vv$}
                \State $\texttt{\ww}\gets$ the word in $\texttt{\WW}_\ll$ equivalent to $\texttt{\vv}$
                \State $\texttt{\ww}\cdot\omega\gets\texttt{\ww}\cdot\omega+\texttt{\uu}\cdot\omega$ 
            \EndIf
		\EndIf
	\EndFor
\EndFor
\State{\textbf{return} $\WW_\ell$}
\EndFunction
\end{algorithmic}
\end{algorithm}

\begin{prop}
Let $\ell\geq2$ be an integer.
Running on sets $\WW_\llo$ and $\WW_\llt$ representing $\BB\nn(\SSS\nn,\llo)$ and $\BB\nn(\SSS\nn,\llt)$ respectively, algorithm \textsc{RepSet} returns a representing set $\WW_\ll$ of  $\BB\nn(\SSS\nn,\ll)$.
Moreover for all $\uu\in \WW_\ell$, the value of~$\emph{\texttt{\uu}}\cdot\omega$ is the integer $\omega_{\SSS\nn}(\overline\uu)$.
\end{prop}

\begin{proof}
Let $\WW'$ be the set of \eqref{E:X} and $\WW_\ll$ be the set returning by \textsc{RepSet}.
The two ``for loops" on line 3 and 4 guarantee $\WW_\ll\subseteq \WW'$. 
By lemma~\ref{L:Geo} and hypotheses on $\WW_\llo$ and $\WW_\llt$, condition $\vv\napp \WW_\llt$ of line 6 tests if the word~$\vv=\uu\xx$ is geodesic.
The second if statement line 7 guarantees we append a word $\vv$ in~$\WW_\ell$ if and only if $\vv$ does not appear in~$\WW_\ell$. The set $W_\ell$ is then a representative set of $\BB\nn(\SSS\nn,\ell)$.
The result about $\omega_{\SSS\nn}$ is a direct consequence of Proposition~\ref{P:NGeo}.
\end{proof}

To be complete we must explain how to test if a $\SSS\nn$-word $\uu$ appears in a set of~$\SSS\nn$-words.
This can be achieved using a normal form (like the Garside's normal form) but such a normal form doesn't provide geodesic representatives. As, for our future research, we want to store braids using geodesic representatives, we prefer to use another method. 

\subsection{Dynnikov's coordinates}

Originally defined in~\cite{Dynnikov} from the geometric interpretation of the braid group~$\BB\nn$ as the mapping class group of the $\nn$-punctured disk of~$\mathbb{R}^2$, the Dynnikov's coordinates admit a purely algebraic definition from the action of $\BB\nn$ on~$\ZZ^{2n}$.

For $x\in\ZZ$, we denote by $x^+$ the non-negative integer $\max(x,0)$ and by $x^-$ the non-positive integer $\min(x,0)$.
We first define an action of Artin's generators on~$\ZZ^4$.

\begin{defi}
For all $\ii\in[1,\nno]$ and all $(x_1,y_1,x_2,y_2)\in\ZZ^4$ we put 
\[
(x_1,y_1,x_2,y_2)\cdot \sig\ii=(x'_1,y'_1,x'_2,y'_2)\quad\text{and}\quad(x_1,y_1,x_2,y_2)\cdot \siginv\ii=(x''_1,y''_1,x''_2,y''_2)
\]
where 
\begin{align*}
x'_1&=x_1+y_1^++(y_2^+-t_1)^+&x''_1&=x_1-y_1^+-(y_2^++t_2)^+\\
y'_1&=y_2-t_1^+&y''_1&=y_2+t_2^-\\
x'_2&=x_2+y_2^-+(y_1^-+t_1)^-&x''_2&=x_2-y_2^--(y_1^--t_2)^-\\
y'_2&=y_1+t_1^+&y''_2&=y_1-t_2^-
\end{align*}
with $t_1=x_1-y_1^--x_2+y_2^+$ and $t_2=x_1+y_1^--x_2-y_2^+$.
\end{defi}

We can now define an action of $\ABP\nn$-words on $\ZZ^{2n}$.

\begin{defi}
For $\ii\in[1,\nno]$, $e=\pm1$ and $(a_1,b_1,\ldots,a_n,b_n)\in\ZZ^{2n}$ we put
\[
(a_1,b_1,\ldots,a_n,b_n)\cdot \sig\ii^e=(a'_1,b'_1,\ldots,a'_n,b'_n)
\]
where $(a'_\ii,b'_\ii,a'_\iip,b'_\iip)=(a_\ii,b_\ii,a_\iip,b_\iip)\cdot \sig\ii^e$ and $a'_k=a_k$, $b'_k=b_k$ for $k$ not belonging to $\{\ii,\iip\}$.
\end{defi}

\begin{defi}
For a $\ABP\nn$-word $\uu$ we define $\Dyn(\uu)$ to be $(0,1,\ldots,0,1)\cdot\uu$.
Similarly for an $\ABKL\nn$-word $\vv$ we define $\Dyn(\vv)$ to be $\Dyn(\uu)$ where $\uu$ is the $\ABP\nn$-word obtained from $\uu$ using relation \eqref{E:apq} of Definition~\ref{D:apq}.
\end{defi}

Naturally defined on braid words, Dynnikov's coordinates is a braid invariant.

\begin{prop}
For all $\SSS\nn$-words $\uu$ and $\vv$ we have $\Dyn(\uu)=\Dyn(\vv)$ if and only if $\uu\equiv\vv$.
\end{prop}

\begin{proof}
Direct consequence of Corollary 2.24 page 225 of \cite{Ordering}.
\end{proof}

We now go back to the problem of testing if a given $\SSS\nn$-word appears in a set~$\WW$ of~$\SSS\nn$-words.
A solution consists in representing the set $\WW$ in machine by an array. 
To test if the word $\uu$ appears in $\WW$ we can compute $\Dyn(\uu)$ and compare it to all the values of $\Dyn(\vv)$ whenever $\vv$ go through $\WW$.
This method needs at most $1+\card{\WW}$ computations of Dynnikov's coordinates. 
If words in $\WW$ are sorted by their Dynnikov's coordinates we can test if $\uu$ appear in $\WW$ using at most $\log_2(\card{\WW})$ computations of Dynnikov's coordinates.
A more efficient solution is obtained using an \texttt{unordered\_set} \cite{STL} based on a hash function.
The insertion and lookup complexity is then constant in  average on a RAM machine depending of the hash function.

As the objective of the current paper is to deepen our knowledge on combinatorics of~$\BB4$, we define a hash function for four strand braids. 
Assume $\br$ is a braid of~$\BB4$ given by a $\SSS4$-word $\uu$.
The hash of $\br$ is 
\[
\hash(\br)=\sum_{i=1}^4 \left(\rem(a_i,256)\times 256^{2i-2}+\rem(b_i,256)\times 256^{2i-1}\right),
\]
where $(a_1,b_1,\ldots,a_4,b_4)=\Dyn(\uu)$ and $\rem(\kk,256)$ is the positive remainder of~$\kk$ modulo~$256$.
By construction, $\hash(\br)$ is an integer lying in $[0,2^{64}-1]$ and so our hash function is very well suited for $64$ bits computers.

\subsection{Space complexity}

Here again we focus on the case $n=4$.
The smallest addressable unit of memory on common computers is the byte which can have $256$ different values.
As the set $\ABP4$ has $6$ elements we can store three $\ABP4$-letters using one byte ($6^3=216$).
Hence a $\ABP4$-word of length $\ell$ requires $\lceil\frac{\ell}3\rceil$ bytes to be stored.
Since there are $12$ elements in $\ABKL4$, a $\ABKL4$-word of length $\ell$ requires $\lceil\frac{\ell}2\rceil$ bytes to be stored.

Assume we want to determine a representative set of $\BB4(\ABP4,21)$. The memory needed by the algorithm \textsc{RepSet} is at least the space needed to store $\ABP4$-words of~$\WW_{21}$.
By Table~\ref{T:B4Artin} of Section~\ref{S:Results} there are approximatively $60\cdot 10^9$ elements in this set.
With the above storage method of a $\ABP4$-word, the algorithm needs $7\cdot60\cdot 10^9$ bytes, \ie, $391$\texttt{Go} of memory to run, which is too much.
To reduce the memory requirement we can split the sets~$\BB\nn(\SSS\nn,\ell)$ in many subsets depending of the values of certain braid invariants.

In case we want to determine $g(\BB\nn,\SSS\nn;\ell)$ we also store the value of $\omega_{\SSS\nn}(\overline{\uu})$ for all words in obtained representative sets. 

\section{Braid template}

\label{S:Template}

Here again $\nn$ is  an integer $\geq 2$ and $\SSS\nn$  denotes either $\ABP\nn$ or $\ABP\nn^\ast$.
Each braid invariant $\iota$ corresponds to a map from $\BB\nn$ to a set $X$.

\begin{defi}
A set of braid invariants $\iota_1,\ldots,\iota_m$ is said to be \emph{inductively stable} if for every braid $\br\in\BB\nn(\SSS\nn,\ell)$ and every $x\in\SSS\nn$, and every $k=1,\ldots,m$,  
the value of $\iota_\kk(\br\cdot x)$ depends only on $\iota_1(\br),\ldots,\iota_m(\br)$ and $x$ but not on $\br$ itself.
\end{defi}

The aim of this section is to determine an inductively stable set of braid invariants in order to split in 
many pieces the determination of a representative set of~$\BB\nn(\SSS\nn,\ell)$.  

\subsection{Permutation}
For $n\geq2$ we denote by $\Sym\nn$ the set of all bijections of $\{1,\ldots,n\}$ into itself.
The transposition $(\ii\  \iip)$ of $\Sym\nn$ exchanging $i$ and $\iip$ is denoted $s_i$.

\begin{defi}
 We denote by $\pi:\BB\nn\to\Sym\nn$ the surjective homomorphism of~$(\BB\nn,\cdot)$ to $(\Sym\nn,\circ)$ defined by
 $\pi(\sig\ii)= s_i$.
\end{defi}

If $\br$ is a braid of $\BB\nn$ then $\pi(\br)$ is the permutation of $\Sym\nn$ such that the strand ending at position $\ii$ starts at position $\pi(\beta)(\ii)$.

\begin{exam}
For $\br=\sig1\siginv2\sig1\sig2$ we have $\pi(\br)=s_1\,s_2\,s_1\,s_2=\big(\!{\tiny\begin{array}{ccc}1&2&3\\3&1&2\end{array}}\!\big)$, as illustrated on the following diagram :

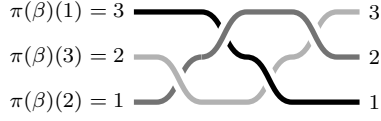
\vspace{0.2em}

\begin{center}
    \begin{tikzpicture}[x=0.06cm,y=0.06cm]
    \draw[line width=2,color=white!45!black] (-5,0) -- (0,0) .. controls (5,0) and (5,10) .. (10,10);
    \draw[line width=6,color=white] (0,10) .. controls (5,10) and (5,0) .. (10,0); 
    \draw[line width=2,color=white!70!black] (-5,10) -- (0,10) .. controls (5,10) and (5,0) .. (10,0) -- (20,0) .. controls (25,0) and (25,10) .. (30,10) .. controls (35,10) and (35,20) .. (40,20) -- (45,20);
    \draw[line width=6,color=white] (20,10) .. controls (25,10) and (25,0) .. (30,0);
    \draw[line width=2] (-5,20) -- (10,20) .. controls (15,20) and (15,10) .. (20,10) .. controls (25,10) and (25,0) .. (30,0) -- (45,0);
    \draw[line width=6,color=white] (10,10) .. controls (15,10) and (15,20) .. (20,20);
    \draw[line width=6,color=white] (30,20) .. controls (35,20) and (35,10) .. (40,10);
    \draw[line width=2,color=white!45!black] (10,10) .. controls (15,10) and (15,20) .. (20,20) -- (30,20) .. controls (35,20) and (35,10) .. (40,10) -- (45,10);
    %\draw[line width=6,color=white] (40,20) .. controls (45,20) and (45,10) .. (50,10);
    %\draw[line width=2,white!40!black] (40,20) .. controls (45,20) and (45,10) .. (50,10) -- (55,10);
    \draw(45,0) node[right]{\footnotesize$1$};
    \draw(45,10) node[right]{\footnotesize$2$};
    \draw(45,20) node[right]{\footnotesize$3$};
    \draw(-5,0) node[left]{\footnotesize$\pi(\beta)(2)=1$};  
    \draw(-5,10) node[left]{\footnotesize$\pi(\beta)(3)=2$};  
    \draw(-5,20) node[left]{\footnotesize$\pi(\beta)(1)=3$};  
    \end{tikzpicture}
  \end{center}
\end{exam}

As $\pi$ is a homomorphism, for all $\br\in\BB\nn$ and $x\in \SSS\nn$ we have $\pi(\br\cdot x)=\pi(\br)\circ \pi(x)$ and so the singleton $\{\pi\}$ is inductively stable.

\begin{lemm}
For $1\leq \indi<\indii\leq \nn$ we have $\pi(\aa\indi\indii)=(\indi\ \indii)$.
\end{lemm}

\begin{proof}
As $\pi$ is a homomorphism, Definition~\ref{D:apq} gives 
\begin{align*}
\pi(\aa\indi\indii)&=\pi(\sig\indi)\circ\ldots\circ\pi(\sig\indiio)\circ\pi(\sig\indiit)\inv\circ\ldots\circ\pi(\sig\indi)\inv\\
&=(\indi\ \indip)\circ\ldots\circ(\indiio\ \indii)\circ (\indiit\ \indiio)\circ \ldots\circ (\indi\ \indip)\\
&=(p\ q).\qedhere
\end{align*}
\end{proof}

\subsection{Linking numbers}

Assume $\br$ is a braid of $\BB\nn$ and let $\ii$ and $\jj$ be two different integers of $[1,\nn]$.
The linking number of the two strands $\ii$ and $\jj$ in $\br$ is the algebraic number of crossings in $\br$ involving the strands $\ii$ and $\jj$.
A positive crossing ($\sig{\kk}$) counts for $+1$ whereas a negative one $(\siginv{\kk})$ counts for $-1$ :

\begin{center}
\begin{tikzpicture}[x=0.07cm,y=0.07cm]
\draw[line width=2] (0,0) .. controls (5,0) and (5,10) .. (10,10);
\draw[line width=6,color=white] (0,10) .. controls (5,10) and (5,0) .. (10,0);
\draw[line width=2] (0,10) .. controls (5,10) and (5,0) .. (10,0);
\draw (18,5) node{$\rightarrow+1$};
 \begin{scope}[shift={(60,0)}]
\draw[line width=2] (0,10) .. controls (5,10) and (5,0) .. (10,0);
\draw[line width=6,color=white] (0,0) .. controls (5,0) and (5,10) .. (10,10);
\draw[line width=2] (0,0) .. controls (5,0) and (5,10) .. (10,10);
\draw (18,5) node{$\rightarrow-1$};
 \end{scope}
\end{tikzpicture}

\end{center}  

\begin{defi}
 For $\br\in\BB\nn$ and $i$, $j$ two different integers of $[1,\nn]$ we denote by~$\ell_{i,j}(\br)$ the linking number of strands $\ii$ and $\jj$ in $\br$.
 The map $\ell_{i,j}:\BB\nn\to\ZZ$ is then a braid invariant.
\end{defi}

A priori, our definition of linking numbers depends of a diagram coding the braid and not on the braid itself. 
An immediate argument using relations \eqref{E:PBn} and \eqref{E:Pssinv} guarantees this is not the case.
The reader can consult \cite{Dehornoy-calcul} page 29 for a more formal definition of linking number\footnote{In fact, the two definitions are slightly different but we have $\ell_{i,j}(\br)=2\lambda_{i,j}(\br)$.}  based of an integral definition and a geometric realization of $\br$ in $\mathbb{R}^3$.

\begin{lemm}
\label{L:sigap}
Let  $\ii,\jj$ be two integers satisfying $1\leq \ii<\jj\leq \nn$ and $e=\pm1$.

\HH -- For all~$\kk\in[1,\nno]$ we have 
\[
\ell_{\ii,\jj}(\sig\kk^e)=\begin{cases}
e & \text{if $\ii=\kk$ and $\jj=\kkp$},\\
0 & \text{otherwise}.
\end{cases}
\]

\HH -- For all~$1\leq \indi<\indii \leq \nn$ we have
\[
\ell_{\ii,\jj}(\aa\indi\indii^e)=\begin{cases}
e & \text{if $\ii=\indi$ and $\jj=\indii$},\\
1 & \text{if $\ii=\indi$ and $\jj<\indii$},\\
-1 & \text{if $\indi<\ii$ and $\jj=\indii$},\\
0 & \text{otherwise}.
\end{cases}
\]
\end{lemm}

\begin{proof}
The case of $\sig\kk^e$ is immediate. 
The different values of $\ell_{\ii,\jj}(\aa\indi\indii^e)$ can be obtained from the following diagram of $\aa\indi\indii^e=\sig\indi\ldots\sig\indiit\ \sig\indiio^e\ \siginv\indiit\ldots\siginv\indi$ : 
\vspace{0.6em}

\hfill  \begin{tikzpicture}[x=0.05cm,y=0.05cm]
     \draw[line width=2] (-5,-10) -- (55,-10);
      \draw[line width=2] (-5,0) -- (0,0) .. controls (5,0) and (5,10) .. (10,10) .. controls (15,10) and (15,20) .. (20,20) .. controls (25,20) and (25,30) .. (30,30) -- (55,30); 
      \draw[line width=6,color=white] (-5,30) -- (20,30) .. controls (25,30) and (25,20) .. (30,20) .. controls (35,20) and (35,10) .. (40,10) .. controls (45,10) and (45,0) .. (50,0) -- (55,0); 
      \draw[line width=2] (-5,30) -- (20,30) .. controls (25,30) and (25,20) .. (30,20) .. controls (35,20) and (35,10) .. (40,10) .. controls (45,10) and (45,0) .. (50,0) -- (55,0); 
      \draw[line width=6,color=white] (-5,10) -- (0,10) .. controls (5,10) and (5,0) .. (10,0) -- (40,0) .. controls (45,0) and (45,10) .. (50,10) -- (55,10);
      \draw[line width=2] (-5,10) -- (0,10) .. controls (5,10) and (5,0) .. (10,0) -- (40,0) .. controls (45,0) and (45,10) .. (50,10) -- (55,10);
      \draw[line width=6,color=white] (-5,20) -- (10,20) .. controls (15,20) and (15,10) .. (20,10) -- (30,10) .. controls (35,10) and (35,20) .. (40,20) -- (55,20);
      \draw[line width=2] (-5,20) -- (10,20) .. controls (15,20) and (15,10) .. (20,10) -- (30,10) .. controls (35,10) and (35,20) .. (40,20) -- (55,20);
      \draw[line width=2] (-5,40) -- (55,40);
      \draw(-5,0) node[left]{\footnotesize$\indi$};
      \draw(-5,30) node[left]{\footnotesize$\indii$};
      \filldraw[color=white](19,19) rectangle (31,31);
      \draw[color=black](19,19) rectangle (31,31);
      \draw(25,25) node{$e$};
    \end{tikzpicture}
    \qedhere
\end{proof}

\begin{lemm}
\label{L:inv}
For $\br$ and $\brbr$ two braids of $\BB\nn$ and $1\leq\ii<\jj\leq\nn$ we have
\[
\ell_{i,j}(\br\cdot \brbr)=\ell_{i,j}(\br)+\ell_{\pi(\br)\inv(\ii),\pi(\br)\inv(\jj)}(\brbr),
\]
with the convention $\ell_{\indi,\indii}=\ell_{\indii,\indi}$ for $\indi>\indii$.
\end{lemm}

\begin{proof}
Immediate as soon as we consider the following diagram :
\vspace{0.6em}

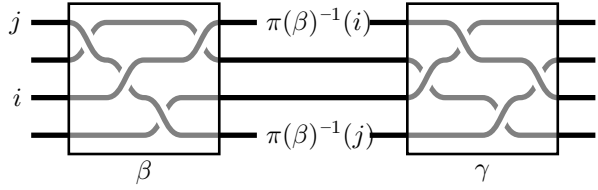
\hfill\begin{tikzpicture}[x=0.05cm,y=0.05cm]
\draw (-10,10) node[left]{$\ii$};
\draw (-10,30) node[left]{$\jj$};
\draw[line width=2] (-10,0) -- (0,0);
\draw[line width=2] (-10,10) -- (0,10);
\draw[line width=2] (-10,20) -- (0,20);
\draw[line width=2] (-10,30) -- (0,30);
\draw[line width=2,color=gray] (0,0) -- (20,0) .. controls (25,0) and (25,10) .. (30,10) -- (40,10);
\draw[line width=2,color=gray] (0,20) -- (0,20) .. controls (5,20) and (5,30) .. (10,30) -- (30,30) .. controls (35,30) and (35,20) .. (40,20);
\draw[line width=6,color=white] (0,30) -- (0,30) .. controls (5,30) and (5,20) .. (10,20) .. controls (15,20) and (15,10) .. (20,10) .. controls (25,10) and (25,0) .. (30,0) -- (40,0);
\draw[line width=2,color=gray] (0,30) -- (0,30) .. controls (5,30) and (5,20) .. (10,20) .. controls (15,20) and (15,10) .. (20,10) .. controls (25,10) and (25,0) .. (30,0) -- (40,0);
\draw[line width=6,color=white] (0,10) -- (10,10) .. controls (15,10) and (15,20) .. (20,20) -- (30,20) .. controls (35,20) and (35,30) .. (40,30);
\draw[line width=2,color=gray] (0,10) -- (10,10) .. controls (15,10) and (15,20) .. (20,20) -- (30,20) .. controls (35,20) and (35,30) .. (40,30);
\draw[line width=1] (0,-5) rectangle (40,35);
\draw[line width=2] (40,0) -- (50,0);
\draw[line width=2] (40,10) -- (90,10);
\draw[line width=2] (40,20) -- (90,20);
\draw[line width=2] (40,30) -- (50,30);
\draw(50,30) node[right]{$\pi(\br)\inv(\ii)$};
\draw(50,0) node[right]{$\pi(\br)\inv(\jj)$};
\draw[line width=2] (80,0) -- (90,0);
\draw[line width=2] (80,30) -- (90,30);
\draw[line width=2,color=gray] (90,20) .. controls (95,20) and (95,10) .. (100,10) -- (110,10) .. controls (115,10) and (115,0) .. (120,0) -- (130,0);
\draw[line width=6,color=white] (90,10) .. controls (95,10) and (95,20) .. (100,20) .. controls (105,20) and (105,30) .. (110,30);
\draw[line width=2,color=gray] (90,10) .. controls (95,10) and (95,20) .. (100,20) .. controls (105,20) and (105,30) .. (110,30) -- (130,30);
\draw[line width=6,color=white] (90,0) -- (110,0) .. controls (115,0) and (115,10)  .. (120,10) .. controls (125,10) and (125,20) .. (130,20);
\draw[line width=2,color=gray] (90,0) -- (110,0) .. controls (115,0) and (115,10)  .. (120,10) .. controls (125,10) and (125,20) .. (130,20);
\draw[line width=6,color=white] (90,30) -- (100,30) .. controls (105,30) and (105,20) .. (110,20) -- (120,20) .. controls (125,20) and (125,10) .. (130,10); 
\draw[line width=2,color=gray] (90,30) -- (100,30) .. controls (105,30) and (105,20) .. (110,20) -- (120,20) .. controls (125,20) and (125,10) .. (130,10); 
\draw[line width=1] (90,-5) rectangle (130,35);
\draw[line width=2] (130,0) -- (140,0);
\draw[line width=2] (130,10) -- (140,10);
\draw[line width=2] (130,20) -- (140,20);
\draw[line width=2] (130,30) -- (140,30);
\draw (20,-10) node{$\br$}; 
\draw (110,-10) node{$\brbr$}; 
\end{tikzpicture}\qedhere
\end{proof}

\begin{coro}
\label{C:InvRS}
The set of invariants $\{\pi\}\cup\{\ell_{\ii,\jj},\ 1 \leq \ii<\jj \leq \nn\}$ is inductively stable. 
\end{coro}

\begin{proof}
A direct consequence of Lemma~\ref{L:inv} together with the fact that $\pi$ is a homomorphism.
\end{proof}

\subsection{Template}

We now introduce the notion of template of a braid which will be used to parallelize the determination of a representative set of $\BB\nn(\SSS\nn,\ell)$.

\begin{defi}
\label{D:Template}
The \emph{template} of a braid $\br\in\BB\nn$ is the tuple
\[
\tau(\br)=(\pi(\br),\ell_{1,2}(\br),\ell_{1,3}(\br),\ell_{2,3}(\br),\ldots,\ell_{1,n}(\br),\ldots,\ell_{\nno,\nn}(\br))\in \Sym\nn\times \ZZ^{\frac{n(n-1)}2},
\]
where integer $\ell_{i,j}(\br)$ appears before $\ell_{r,s}(\br)$ whenever $(i,j)$ is smaller than $(r,s)$ with respect to the the co-lexicographic ordering on $\NN^2$  : $(i,j)<(p,q)$ if $j<q$ or if $j=q$ and $i<p$.
For a braid template $t$ we denote by $t[\pi]$, \resp $t[\ell_{\ii,\jj}]$ the corresponding component.
For $\ell\in\NN$ we denote by~$\TT\nn(\SSS\nn,\ell)$ the set $\{\tau(\br),\ \br\in\BB\nn(\SSS\nn,\ell)\}$ and by $\TT\nn$ the set~$\{\tau(\br),\ \br\in\BB\nn\}$ of all templates on $\BB\nn$.

\end{defi}

\begin{lemm}
For all $\br\in\BB\nn$ and all $x\in\SSS\nn$, the template $\tau(\br)\ast x=\tau(\br\cdot x)$ depends only on~$\tau(\br)$ and $x$.
\end{lemm}

\begin{proof}
A direct consequence of Corollary~\ref{C:InvRS} and Definition~\ref{D:Template}.
\end{proof}

\begin{exam}
Let $t$ be a template of $T_3$ with $t[\pi]$ the cycle $(1\ 3\ 2)$. 
Let us compute the template $t\ast \aa13\inv$. We write $t=(\pi,\ell_{1,2},\ell_{1,3},\ell_{2,3})$.
The inverse of $\pi$ is the cycle $(1\ 2\ 3)$ and so we obtain $\pi\inv(\{1,2\})=\{2,3\}$, $\pi\inv(\{1,3\})=\{1,2\}$ and $\pi\inv(\{2,3\})=\{1,3\}$.
Eventually, from $\ell_{1,2}(\aa13\inv)=1$, $\ell_{1,3}(\aa13\inv)=-1$ and $\ell_{2,3}(\aa13\inv)=-1$ we obtain
\begin{align*}
t\ast \aa13\inv&=\left( (1\ 3\ 2)\circ (1\ 3), \ell_{1,2}+\ell_{2,3}(\aa13\inv), \ell_{1,3}+\ell_{1,2}(\aa13\inv), \ell_{2,3}+\ell_{1,3}(\aa13\inv)\right)\\
&=\left((1\ 2),\ell_{1,2}-1,\ell_{1,3}+1,\ell_{2,3}-1\right).
\end{align*}
\end{exam}

\begin{defi}
\label{D:GeoSet}
For $\ell\in \NN$ and $t\in \TT\nn$ we denote by $\BB\nn(\SSS\nn,\ell,t)$ the set of all braids of $\BB\nn$ with $\SSS\nn$-length $\ell$ and template $t$.
\end{defi} 

By very definitions we have
\begin{equation}
\BB\nn(\SSS\nn,\ell)=\bigsqcup_{t\in \TT\nn(\SSS\nn,\ell)} \BB\nn(\SSS\nn,\ell,t).
\end{equation}

Algorithm~\ref{A:TempRepSet} -- \textsc{TempRepSet} is  a ``template'' version of Algorithm~\ref{A:Rep} -- \textsc{RepSet} for which we assume we dispose of a function $\textsc{Load}(\nn,\ell,t)$ loading a representative set of $\BB\nn(\SSS\nn,\ell,t)$ from a storage memory like a hard disk.
 We also assume we have a function $\textsc{Save}(\WW,\nn,\ell,t)$ saving a representative set of~$\BB\nn(\SSS\nn,\ell,t)$ to that storage memory.
 
\begin{algorithm}[h]
\caption{-- \small \textsc{TempRepSet} : For an integer $\ell \geq 1$ and a template $t$ of $\TT\nn(\SSS\nn,\ell)$, saves a representative set $\texttt{\WW}_{\ll,t}$ of $\BB\nn(\SSS\nn,\ell,t)$ and returns the pair $(\card{\texttt{\WW}_\ell},\sum_{\uu\in\texttt{\WW}_\ell} \omega_{\SSS\nn}(\overline\uu))$}
\label{A:TempRepSet}
\small
\begin{algorithmic}[1]
\Function{TempRepSet}{$\ell,\,t$}
\State $\texttt{\WW}_{\ll,t}\gets \emptyset$
\State $\texttt{\WW}_{\llt,t}\gets \textsc{Load}(\nn,\llt,t)$ \Comment{$\texttt{\WW}_{\llt,t}$ is empty whenever $\ell=1$.}
\State $n_\texttt{g}\gets 0$
\For{$\texttt{x}\in \SSS\nn$}
	\State $t_\texttt{x}\gets t\ast \texttt{x}\inv$
	\State $\texttt{\WW}_{\llo,\texttt{x}}\gets \textsc{Load}(\nn,\llo,t_\texttt{x})$
	\For{$\texttt{\uu}\in \texttt{\WW}_{\llo,\texttt{x}}$}	
        \State $\texttt{\vv}\gets \texttt{\uu}\,\texttt{\xx}$
		\If{$\texttt{\vv}\napp \texttt{\WW}_{\llt,t}$}
            \If{$\texttt{\vv}\napp\texttt{\WW}_{\ll,t}$}
                \State $\texttt{\WW}_{\ell,t}\gets\texttt{\WW}_{\ell,t}\sqcup\{\texttt{\vv}\}$
                \State $\texttt\vv\cdot\omega\gets\texttt{u}\cdot\omega$
            \Else
             \State $\texttt{\ww}\gets$ the word in $\texttt{\WW}_{\ll,t}$ equivalent to $\texttt{\vv}$
            \State $\texttt\ww\cdot\omega\gets\texttt\ww\cdot\omega+\texttt{u}\cdot\omega$
            \EndIf
             \State $n_\texttt{g}\gets n_\texttt{g}+\texttt{u}\cdot\omega$
		\EndIf
	\EndFor
\EndFor
\State $\textsc{Save}(\texttt{W}_{\ell,t},n,\ell,t)$
\State \textbf{return} $(\card{\texttt{W}_{\ell,t}},n_\texttt{g})$
\EndFunction
\end{algorithmic}
\end{algorithm}

In order to compute a representative set of $\BB\nn(\SSS\nn,\ell)$ using Algorithm \textsc{TempRepSet} 
we must first compute the template set $\TT\nn(\SSS\nn,\ell)$. 
From inclusion \eqref{E:B:sub} we obtain
\begin{equation}
\label{E:TempSub}
\TT\nn(\SSS\nn,\ll)\subseteq \{ t\ast x\ \text{for}\ (t,x)\in \TT\nn(\SSS\nn,\llo)\times \SSS\nn\}.
\end{equation}
A template $t$ from the set in the right-hand side of  \eqref{E:TempSub} belongs to $\TT\nn(\SSS\nn,\ll)$ if and only
if there exists a braid $\br\in \BB\nn(\SSS\nn,\ell)$ such that $\tau(\br)=t$.
Hence a full run consists in calling the function \textsc{TempRepSet} for each template $t$ from the set in the right-hand side of~\eqref{E:TempSub}.
Such a template $t$ will belongs to~$\TT\nn(\SSS\nn,\ll)$ if and only if the returned value is different from $(0,0)$.
Putting all pieces together we obtain :

\begin{algorithm}
\caption{-- \small \textsc{Combi} : Returns a pair of arrays of positive integers $(n_\texttt{s},n_\texttt{g})$ satisfying relations  $n_\texttt{s}[\ell]=s(\BB\nn,\SSS\nn;\ell)$ and $n_\texttt{g}[\ell]=g(\BB\nn,\SSS\nn;\ell)$ for all $\ell\leq\ell_{\text{max}}$.}
\label{A:NumB}
\small
\begin{algorithmic}[1]
\Function{Combi}{$\ell_{\text{max}}$}
\State $n_\texttt{s}[0]\gets 1$
\State $n_\texttt{g}[0]\gets 1$
\State $T\gets \{(1_{\Sym\nn}, 0, \ldots, 0)\}$ \Comment{template set $\TT\nn(\SSS\nn,0)$}
\For{$\ell$ \textbf{from} $1$ \textbf{to} $\ell_{\text{max}}$}
	\State $T'\gets\emptyset$
	\State $n_\texttt{s}[\ell]\gets 0$;\ $n_\texttt{g}[\ell]\gets 0$
	\For{$t\in T$}
		\For{$\texttt{x}\in \SSS\nn$}
			\State $t_\texttt{x}\gets t\ast \texttt{x}$
			\State $(n'_\texttt{s},n'_\texttt{g})\gets \textsc{TempRepSet}(\ell,t_\texttt{x})$
			\If{$(n'_\texttt{s},n'_\texttt{g})\not=(0,0)$}
				\State $T'\gets T'\cup\{t'\}$
				\State $n_\texttt{s}[\ell]\gets n_\texttt{s}[\ell]+n_\texttt{s}'$
				\State $n_\texttt{g}[\ell]\gets n_\texttt{g}[\ell]+n'_\texttt{g}$
			\EndIf
		\EndFor
	\EndFor
	\State $T\gets T'$\Comment{template set $\TT\nn(\SSS\nn,\ell)$}
\EndFor
\State \textbf{return} $(n_\texttt{s},n_\texttt{g})$
\EndFunction
\end{algorithmic}
\end{algorithm}

 \section{Reduced braid templates}

 \label{S:Reduced}

 Here again $n$ is an integer $\geq2$ and $\SSS\nn$ denotes either $\ABP\nn$ or $\ABKL\nn$.
 Experiments using Algorithm~\ref{A:TempRepSet} -- \textsc{TempRepSet} suggest that some sets~$\BB\nn(\SSS\nn,\ell,t)$ are in bijection for a given $\ell$.
 We can use this fact to improve the efficiency of Algorithm~\ref{A:NumB} -- \textsc{Combi} and reduce the needed storage space.

 \subsection{Stable word maps}
 
 \begin{defi}
 \label{D:Stable}
 A bijection $\mu$ of the set of $\SSS\nn$-words is \emph{$\SSS\nn$-stable} if
 
 \HH -- $i)$ for all $\SSS\nn$-word $\ww$ we have $|\mu(\ww)|=|\ww|$;
 
 \HH -- $ii)$ for all $\SSS\nn$-words $\uu$ and $\vv$ we have $\mu(\uu)\equiv\mu(\vv)\Leftrightarrow \uu\equiv\vv$;
 
 \HH -- $iii)$ for all $\SSS\nn$-word $\uu$ the template $\tau(\overline{\mu(\uu)})$ depends only on $\tau(\overline{\uu})$.
 
 \noindent For such a $\SSS\nn$-stable map $\mu$ we denote by $\mu^T$ the map of $\TT\nn$ defined by 
 \[
 \mu^T(t)=\tau(\overline{\mu(\uu)})
 \]
 where $\uu$ is any $\SSS\nn$-word satisfying $\tau(\overline{\uu})=t$.
 We also define a bijection $\overline\mu$ of $\BB\nn$ by 
 \[
 \overline\mu(\br)=\overline{\mu(\uu)},
 \] 
 where $\uu$ is any $\SSS\nn$-word satisfying $\overline{\uu}=\br$.
 \end{defi}

Whenever $\mu$ is $\SSS\nn$-stable, Condition $iii)$ of Definition~\ref{D:Stable} guarantees that the template of the image by $\overline{\mu}$ of a braid $\br$ does not depend on $\br$ but on its template~$t$ and so $\mu^T$ is well defined.

\begin{lemm}
\label{L:Stable}
For every $\SSS\nn$-stable bijection $\mu$, we have

\HH -- $i)$ $\mu^T$ is a permutation of $\TT\nn$,

\HH -- $ii)$ a $\SSS\nn$-word $\uu$ is geodesic if and only if $\mu(\uu)$ is.
 \end{lemm}
 
 \begin{proof}
By $i)$ and $ii)$ of Definition~\ref{D:Stable} we obtain that $\mu$ induces a permutation on the finite set $\BB\nn(\SSS\nn,\ell)$. It follows that $\mu^T$ induces a permutation on $\TT\nn(\SSS\nn,\ell)$. For a template $t$ of $T_n$ there exists an integer $\ell\in\NN$ such that $t$ belongs to $T_n(\SSS\nn,\ell)=\mu^T\left(T_n(\SSS\nn,\ell)\right)$ and so $\mu^T$ is surjective.
We now prove the injectivity. For $t\in \TT\nn$, we denote by $\lambda(t)$ the minimal integer $\ell$ such that $t$ belongs to $\TT\nn(\SSS\nn,\ell)$.
Since~$\mu^T$ induces a permutation on  $\TT\nn(\SSS\nn,\ell)$ for all $\ell$ we have $\lambda(\mu^T(t))=\lambda(t)$.
Let~$t$ and $t'$ be two templates of $\TT\nn$ satisfying $\mu^T(t)=\mu^T(t')$. By the above we have 
\[\lambda(t)=\lambda(\mu^T(t))=\lambda(\mu^T(t'))=\lambda(t'),\]
and so there exists $\ell$ such that $t$ and $t'$ belong to $\TT\nn(\SSS\nn,\ell)$.
Since~$\mu^T$ induces a permutation on $T_\nn(\SSS\nn,\ell)$ we obtain $t=t'$, proving the injectivity of $\mu^T$.

Let us now prove $ii)$. Let $\uu$ be a $\SSS\nn$-word.
If the word $\vv=\mu(\uu)$ is not geodesic then there exists a strictly shorter $\SSS\nn$-word $\vv'$ equivalent to $\vv$.
As $\mu$ is a bijection we put $\uu'=\mu\inv(\vv')$.
We obtain $\mu(\uu)=\vv\equiv\vv'=\mu(\uu')$.
From conditions $ii)$ and $i)$ of Definition~\ref{D:Stable} we have $\uu\equiv\uu'$ together with $|\uu|=|\vv|>|\vv'|=|\uu'|$ and so $\uu$ is not geodesic.
A similar argument establishes the converse implication.
 \end{proof}

\subsection{Examples}

Let us now introduce some useful examples of $\SSS\nn$-stable bijections. 
Eventually such a $\SSS\nn$-stable bijection $\mu$ will be used to obtain a representative set of~$\TT\nn(\SSS\nn,\ell,\mu^T(t))$ from a representative set of $\TT\nn(\SSS\nn,\ell,t)$. This is why it is necessary to specify how to obtain $\mu^T(t)$ from $t$ in propositions~\ref{P:inv:stable}, \ref{P:theta:stable}, \ref{P:PhiAcc} and~\ref{P:phi:stable}.
However the reader may choose to ignore these parts without affecting the understanding of the rest ot the article.

\subsubsection{First examples}

\begin{prop}
\label{P:inv:stable}
The map $\mathrm{inv}_{\SSS\nn}$ of $\SSS\nn$-words defined by 
\[
\mathrm{inv}_{\SSS\nn}(x_1\cdots x_t)=x_t\inv\cdots x_1\inv
\]
 is $\SSS\nn$-stable.
 Moreover for every template $t\in\TT\nn$  we have
 \[
 \mathrm{inv}_{\SSS\nn}^T(t)[\pi]=t[\pi]\inv\quad\text{and}\quad\mathrm{inv}_{\SSS\nn}^T(t)[\ell_{\ii,\jj}]=-t[\ell_{t[\pi](\ii),t[\pi](\jj)}]\ \ \text{for $1\leq\ii<\jj\leq\nn$.}
 \]
\end{prop}

\begin{proof}

Condition $i)$ of Definition~\ref{D:Stable} is immediate. 
For two $\SSS\nn$-words $\uu$ and~$\vv$, the relation $\uu\equiv\vv$ is equivalent to $\vv\inv\uu\equiv\ew$ which is itself equivalent to $\vv\inv\equiv\uu\inv$, hence Condition $ii)$ is established. 
Let $\uu$ be a $\SSS\nn$-word and~$\vv$ be $\mathrm{inv}_{\SSS\nn}(\uu)$. 
By definition, we have $\overline{\vv}=\overline{\uu}\inv$.
Since $\pi$ is a homomorphism we have $\pi(\overline{\vv})=\pi(\overline{\uu})\inv$.
Let $1\leq\ii<\jj\leq\nn$ be two integers. 
From $1=\overline{\vv}\overline{\uu}$, Lemma~\ref{L:inv} implies
\[
0=\ell_{\ii,\jj}(1)=\ell_{\ii,\jj}(\overline{\vv})+\ell_{\pi(\overline{\uu})(\ii),\pi(\overline{\uu})(\jj)}(\overline{\uu})
\]
and so $\ell_{\ii,\jj}(\overline{\vv})=-\ell_{\pi(\overline{\uu})(\ii),\pi(\overline{\uu})(\jj)}(\overline{\uu})$.
Therefore Condition $iii)$ is also satisfied.
\end{proof}

We now point out a divergence between the Artin and dual presentations of the braid group $\BB\nn$.

\begin{prop}
\label{P:theta:stable}
For $\nn\geq 3$, the map of $\SSS\nn$-words $\theta_{\SSS\nn}$ defined by 
\[
\theta_{\SSS\nn}(x_1\cdots x_t)=x_1\inv\cdots x_t\inv
\]
is $\SSS\nn$-stable if and only if $\SSS\nn=\ABP\nn$.
 Moreover for every template $t\in\TT\nn$  we have
 \[
 \theta_{\ABP\nn}^T(t)[\pi]=t[\pi]\quad\text{and}\quad
  \theta_{\ABP\nn}^T(t)[\ell_{\ii,\jj}]=-t[\ell_{\ii,\jj}]\ \ \text{for $1\leq\ii<\jj\leq\nn$.}
 \]
\end{prop}

\begin{proof} 
By construction, Condition $i)$ of Definition \ref{D:Stable} is satisfied.
Let us verify Condition~$ii)$ for $\theta_{\ABP\nn}$.
It is sufficient to prove $\theta(\uu)\equiv\theta(\vv)$ whenever $\uu=\vv$ is 
a relation of the Artin's semigroup presentation of $\BB\nn$.
Let $\ii\in[1,\nno]$.
We have~$\theta(\sig\ii\siginv\ii)=\siginv\ii\sig\ii\equiv \ew$,  
$\theta(\siginv\ii\sig\ii)=\sig\ii\siginv\ii\equiv \ew$ and so we get 
\[
\theta(\sig\ii\siginv\ii)=\theta(\siginv\ii\sig\ii)=\theta(\ew).
\]
Assume now $\ii$ and $\jj$ are integers of $[1,\nno]$ satisfying $|\ii-\jj|\geq2$.
From $\sig\ii\sig\jj\equiv\sig\jj\sig\ii$ we obtain successively 
\[
\siginv\jj\sig\ii\sig\jj\equiv\sig\ii,\quad  \siginv\jj\sig\ii\equiv\sig\ii\siginv\jj,\quad  \siginv\ii\siginv\jj\sig\ii\equiv\siginv\jj,\quad \siginv\ii\siginv\jj\equiv\siginv\jj\siginv\ii,
\] 
and so $\theta(\sig\ii\sig\jj)=\siginv\ii\siginv\jj\equiv\siginv\jj\siginv\ii=\theta(\sig\jj\sig\ii)$.
A similar sequence of equivalences implies $\theta(\sig\ii\sig\jj\sig\ii)\equiv\theta(\sig\jj\sig\ii\sig\jj)$ for $\ii,\jj$ in $[1,\nno]$ with $|\ii-\jj|\leq 1$.

Let $\uu$ be an $\ABP\nn$-word.
For $\xx\in\ABP\nn$, the permutation~$\pi(\overline{x})$ is a transposition and so the relation $\pi(\overline{x})=\pi(\overline{x\inv})$ holds. 
Hence we obtain~$\pi(\overline{\theta_{\ABP\nn}(\uu)})=\pi(\overline{\uu})$.
We denote by $\uu_k$ the prefix of~$\uu$ of length $\kk$.
An immediate induction on $\kk$, together with~$\pi(\overline{\theta_{\ABP\nn}(\uu_\kk)})=\pi(\overline{\uu_\kk})$ and Lemma~\ref{L:inv} establish $\ell_{i,j}(\overline{\theta_{\ABP\nn}(\uu)})=-\ell_{i,j}(\overline{\uu})$.
Condition $iii)$ is then satisfied by $\theta_{\ABP\nn}$.

Let us focus now on the map $\theta_{\ABKL\nn}$.
In $\BB4$ we have the relation $\aa12\aa23\equiv\aa23\aa13$ while $\aa12\inv\aa23\inv$ is not equivalent to $\aa23\inv\aa13\inv$ as shown by the following diagrams. 
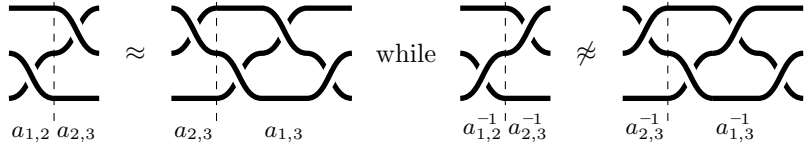
\begin{center}
\begin{tikzpicture}[x=0.06cm,y=0.06cm]
\draw[line width=2] (0,0) .. controls (5,0) and (5,10) .. (10,10) .. controls (15,10) and (15,20) .. (20,20);
\draw[line width=6,color=white] (0,10) .. controls (5,10) and (5,0) .. (10,0) -- (20,0);
\draw[line width=2] (0,10) .. controls (5,10) and (5,0) .. (10,0) -- (20,0);
\draw[line width=6,color=white] (0,20) -- (10,20) .. controls (15,20) and (15,10) .. (20,10);
\draw[line width=2] (0,20) -- (10,20) .. controls (15,20) and (15,10) .. (20,10);
\draw(5,-8) node{\small$\aa12$};
\draw(15,-8) node{\small$\aa23$};
\draw[dashed](10,-5) -- (10,23);
\begin{scope}[shift={(36,0)}]
%2
\draw[line width=2] (0,10) .. controls (5,10) and (5,20) .. (10,20);

%3
\draw[line width=6,color=white] (0,20) .. controls (5,20) and (5,10) .. (10,10);
\draw[line width=2] (0,20) .. controls (5,20) and (5,10) .. (10,10) ;

\draw[line width=6,color=white] (0,0) -- (10,0) .. controls (15,0) and (15,10) .. (20,10) .. controls (25,10) and (25,20) .. (30,20) -- (40,20);
%1 
\draw[line width=2] (0,0) -- (10,0) .. controls (15,0) and (15,10) .. (20,10) .. controls (25,10) and (25,20) .. (30,20) -- (40,20);

%2 bis
\draw[line width=6,color=white] (10,20) -- (20,20) .. controls (25,20) and (25,10) .. (30,10) .. controls (35,10) and (35,0) .. (40,0);
\draw[line width=2] (10,20) -- (20,20) .. controls (25,20) and (25,10) .. (30,10) .. controls (35,10) and (35,0) .. (40,0);

%3 bis
\draw[line width=6,color=white] (10,10) .. controls (15,10) and (15,0) .. (20,0) -- (30,0) .. controls (35,0) and (35,10) .. (40,10);
\draw[line width=2] (10,10) .. controls (15,10) and (15,0) .. (20,0) -- (30,0) .. controls (35,0) and (35,10) .. (40,10);

\draw(5,-8) node{\small$\aa23$};
\draw(25,-8) node{\small$\aa13$};
\draw[dashed](10,-5) -- (10,23);
\end{scope}
\draw(28,10) node{$\approx$};

\begin{scope}[shift={(100,0)}]

\draw[line width=2] (0,10) .. controls (5,10) and (5,0) .. (10,0) -- (20,0);
\draw[line width=6,color=white] (0,20) -- (10,20) .. controls (15,20) and (15,10) .. (20,10);
\draw[line width=2] (0,20) -- (10,20) .. controls (15,20) and (15,10) .. (20,10);
\draw[line width=6,color=white] (0,0) .. controls (5,0) and (5,10) .. (10,10) .. controls (15,10) and (15,20) .. (20,20);
\draw[line width=2] (0,0) .. controls (5,0) and (5,10) .. (10,10) .. controls (15,10) and (15,20) .. (20,20);
\draw(5,-6.5) node{\small$\aa12\inv$};
\draw(15,-6.5) node{\small$\aa23\inv$};
\draw[dashed](10,-5) -- (10,23);
\begin{scope}[shift={(36,0)}]
%1
\draw[line width=2] (0,0) -- (10,0) .. controls (15,0) and (15,10) .. (20,10);
%3
\draw[line width=6,color=white] (0,20) .. controls (5,20) and (5,10) .. (10,10) .. controls (15,10) and (15,0) .. (20,0);
\draw[line width=2] (0,20) .. controls (5,20) and (5,10) .. (10,10) .. controls (15,10) and (15,0) .. (20,0);
%2
\draw[line width=6,color=white] (0,10) .. controls (5,10) and (5,20) .. (10,20) -- (20,20) .. controls (25,20) and (25,10) .. (30,10) .. controls (35,10) and (35,0) .. (40,0);
\draw[line width=2] (0,10) .. controls (5,10) and (5,20) .. (10,20) -- (20,20) .. controls (25,20) and (25,10) .. (30,10) .. controls (35,10) and (35,0) .. (40,0);
%3 bis 
\draw[line width=6,color=white] (20,0) -- (30,0) .. controls (35,0) and (35,10) .. (40,10);
\draw[line width=2] (20,0) -- (30,0) .. controls (35,0) and (35,10) .. (40,10);
%1 bis
\draw[line width=6,color=white] (20,10) .. controls (25,10) and (25,20) .. (30,20) -- (40,20);
\draw[line width=2] (20,10) .. controls (25,10) and (25,20) .. (30,20) -- (40,20);
\draw(5,-6.5) node{\small$\aa23\inv$};
\draw(25,-6.5) node{\small$\aa13\inv$};
\draw[dashed](10,-5) -- (10,23);
\end{scope}
\draw(28,10) node{$\not\approx$};
\end{scope}

\draw(88,10) node{while};
\end{tikzpicture}
\end{center}
The non isotopy of the two right-most diagrams can be established evaluating~$\ell_{1,3}$ for example.
Indeed we have $\ell_{1,3}(\aa12\inv\aa23\inv)=-1$ and $\ell_{1,3}(\aa23\inv\aa13\inv)=1$.
\end{proof}
 
 \subsubsection{Garside homorphisms}
 
 We now consider the ``word version" of the classical and dual Garside automorphisms of $\BB\nn$.
 
 \begin{defi} 
The Garside automorphism of $\BB\nn$ is  $\overline{\Phi}_\nn(\br)=\Delta_\nn\,\br\,\Delta_\nn\inv$ where $\Delta_\nn$ is given by $\Delta_2=\sig1$ and $\Delta_\kk=\sig1\cdots\sig\kko\Delta_\kko$ for $\kk\geq 3$.
 \end{defi}
 
For example we have $\Delta_4=\sig1\sig2\sig3\cdot\Delta_3=\sig1\sig2\sig3\cdot\sig1\sig2\cdot\Delta_2=\sig1\sig2\sig3\cdot\sig1\sig2\cdot\sig1$, which corresponds to the following diagram:
 
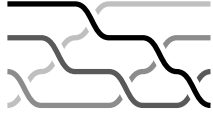
\begin{center}
 \begin{tikzpicture}[x=0.045cm,y=0.045cm]
  \draw[line width=2,color=white!75!black] (0,0) .. controls (5,0) and (5,10) .. (10,10) .. controls (15,10) and (15,20) .. (20,20) .. controls (25,20) and (25,30) .. (30,30) -- (60,30);
  
  \draw[line width=6,color=white] (0,10) .. controls(5,10) and (5,0) .. (10,0) -- (30,0) .. controls (35,0) and (35,10) .. (40,10) .. controls (45,10) and (45,20) .. (50,20) -- (60,20);
  \draw[line width=2,color=white!55!black] (0,10) .. controls(5,10) and (5,0) .. (10,0) -- (30,0) .. controls (35,0) and (35,10) .. (40,10) .. controls (45,10) and (45,20) .. (50,20) -- (60,20);
  
  \draw[line width=6,color=white] (0,20) -- (10,20) .. controls (15,20) and (15,10) .. (20,10) -- (30,10) .. controls (35,10) and (35,0) .. (40,0) -- (50,0) .. controls(55,0) and (55,10) .. (60,10);
  \draw[line width=2,color=white!35!black] (0,20) -- (10,20) .. controls (15,20) and (15,10) .. (20,10) -- (30,10) .. controls (35,10) and (35,0) .. (40,0) -- (50,0) .. controls(55,0) and (55,10) .. (60,10);

  \draw[line width=6,color=white] (0,30) -- (20,30) .. controls (25,30) and (25,20) .. (30,20) -- (40,20) .. controls (45,20) and (45,10) .. (50,10) ..controls(55,10) and (55,0) .. (60,0); 
  \draw[line width=2] (0,30) -- (20,30) .. controls (25,30) and (25,20) .. (30,20) -- (40,20) .. controls (45,20) and (45,10) .. (50,10) ..controls(55,10) and (55,0) .. (60,0); 
   \end{tikzpicture}
 \end{center}
For all $\kk\in[1,\nn]$ we have:
\begin{equation}
\label{E:Delta:pi}
\pi(\Delta_\nn)(\kk)=\nnp-\kk.
\end{equation}
As we can notice in the previous diagram, the braid $\Delta_\nn$ can be represented by a diagram in which each two strands cross exactly once implying
\begin{equation}
\label{E:Delta:ell}
 \ell_{i,j}(\Delta_\nn)=1\quad\text{and}\quad  \ell_{i,j}(\Delta_\nn\inv)=-1\quad\text{for all $1\leq i<j\leq n$}.
\end{equation}
The result involving $\Delta_\nn\inv$ is a direct consequence of that of $\Delta_\nn$ together with Proposition~\ref{P:inv:stable}.
The following lemma is a well-known result about the Garside automorphism $\overline{\Phi}_\nn$.

 \begin{lemm}
 \label{L:Phi}
 For $\nn\geq 3$, the automorphism $\overline{\Phi}_\nn$ has order $2$ and for every integer~$\kk$ in $[1,\nno]$ we have $\overline{\Phi}_\nn(\sig\kk)=\sig{\nn-\kk}$.
\end{lemm}

\begin{proof}
Let $\kk\in[1,\nno]$.
Relation $\overline{\Phi}_\nn(\sig\kk)=\sig{\nn-\kk}$ is an easy verification from the Artin presentation of $\BB\nn$ (see Lemma~I.3.6 of \cite{Dehornoy-calcul}). 
We conclude with $\overline{\Phi}^2_\nn(\sig\kk)=\overline{\Phi}_\nn(\sig{\nn-\kk})=\sig{\nn-(\nn-\kk)}=\sig\kk$.
\end{proof}

\begin{defi}
 We denote by $\Phi_\nn$ the homomorphism of $\ABP\nn$-words defined for every integer  $\kk$ in $[1,\nn]$ by $\Phi_\nn(\sig\kk)=\sig{\nn-\kk}$.
 \end{defi}
 
 By Lemma~\ref{L:Phi}, for every $\ABP\nn$-word $\uu$ we have
\begin{equation}
\label{E:Phi}
\overline{\Phi_\nn(\uu)}=\overline{\Phi}_\nn(\overline{\uu})=\Delta_\nn\,\overline{\uu}\,\Delta_\nn\inv.
\end{equation}

\begin{prop}
\label{P:PhiAcc}
The map $\Phi_\nn$ is $\ABP\nn$-stable.
Moreover for every template $t\in\TT\nn$ we have:
\begin{align*}
\left(\Phi_\nn^T(t)[\pi]\right)(\kk)=n+1-t[\pi](n+1-k)\quad \text{for all $k\in[1,\nn]$},\\
\Phi_\nn^T(t)[\ell_{\ii,\jj}]=t[\ell_{n+1-j,n+1-i}]\quad\text{for all $1\leq\ii<\jj\leq\nn$}.
\end{align*}
\end{prop} 
 
\begin{proof}
Condition $i)$ and $ii)$ of Definition~\ref{D:Stable} are easily established since $\Phi_\nn$ is defined from the conjugation $\overline{\Phi}_\nn$ which induces a bijection on $\ABP\nn$.
We now prove Condition~$iii)$ of Definition~\ref{D:Stable}.
Let $\uu$ be a $\ABP\nn$-word.
By \eqref{E:Phi} we have $$\pi(\overline{\Phi_\nn(\uu)})=\pi(\Delta_\nn)\circ \pi(\overline{\uu})\circ \pi(\Delta_\nn)\inv=\pi(\Delta_\nn)\circ \pi(\overline{\uu})\circ \pi(\Delta_\nn).$$
Relation~\eqref{E:Delta:pi} implies that for every integer $\kk\in[1,n]$ we have
\begin{equation}
\pi(\overline{\Phi_\nn(\uu)})(\kk)=\nnp-\pi(\overline{\uu})(\nnp-k).
\end{equation}
Let $1\leq\ii<\jj\leq\nn$ be two integers.
Lemma~\ref{L:inv} together with \eqref{E:Delta:ell} give
\begin{align*}
\ell_{\ii,\jj}(\overline{\Phi_\nn(\uu)})=\ell_{\ii,\jj}(\Delta_\nn\overline{\uu}\Delta_\nn\inv)&=\ell_{\ii,\jj}(\Delta_\nn\overline{\uu})+\ell_{\pi(\Delta_\nn\overline{\uu})\inv(\ii),\pi(\Delta_\nn\overline{\uu})\inv(\jj)}(\Delta_\nn\inv)\\
&=\ell_{\ii,\jj}(\Delta_\nn\overline{\uu})-1\\
&=\ell_{\ii,\jj}(\Delta_\nn)+\ell_{\pi(\Delta_\nn)\inv(\ii),\pi(\Delta_\nn)\inv(\jj)}(\overline\uu)-1\\
&=\ell_{\pi(\Delta_\nn)\inv(\ii),\pi(\Delta_\nn)\inv(\jj)}(\overline\uu)\\
&=\ell_{\nnp-\ii,\nnp-\jj}(\overline{\uu})=\ell_{\nnp-\jj,\nnp-\ii}(\overline{\uu}).\qedhere
\end{align*}
 \end{proof}

 \begin{defi} 
The dual Garside automorphism of $\BB\nn$ is  $\overline{\phi}_\nn(\br)=\delta_\nn\,\br\,\delta_\nn\inv$ where~$\delta_\nn$ is given by $\delta_\nn=\aa12\cdots\aa\nno\nn=\sig1\cdots\sig\nno$.
 \end{defi}
 
For example we have $\delta_4=\sig1\sig2\sig3$ and $\delta_4\inv=\siginv3\siginv2\siginv1$ which correspond to the following diagrams:

 \begin{equation}
 \label{E:delta}
 \begin{tikzpicture}[x=0.045cm,y=0.045cm,baseline=20]
  \draw[line width=2,color=white!75!black] (0,0) .. controls (5,0) and (5,10) .. (10,10) .. controls (15,10) and (15,20) .. (20,20) .. controls (25,20) and (25,30) .. (30,30);

  \draw[line width=6,color=white] (0,30) -- (20,30) .. controls (25,30) and (25,20) .. (30,20);
  \draw[line width=2] (0,30) -- (20,30) .. controls (25,30) and (25,20) .. (30,20);
  
  \draw[line width=6,color=white] (0,10) .. controls(5,10) and (5,0) .. (10,0) -- (30,0);
  \draw[line width=2,color=white!55!black] (0,10) .. controls(5,10) and (5,0) .. (10,0) -- (30,0);
 
  \draw[line width=6,color=white] (0,20) -- (10,20) .. controls (15,20) and (15,10) .. (20,10) -- (30,10);
  \draw[line width=2,color=white!35!black] (0,20) -- (10,20) .. controls (15,20) and (15,10) .. (20,10) -- (30,10);
  \draw(-12,15) node{$\delta_4\simeq$};
  \begin{scope}[shift={(80,0)}]
    \draw(-13,15) node{$\delta_4\inv\simeq$};
    \draw[line width=2,color=white!75!black] (0,30) .. controls (5,30) and (5,20) .. (10,20) .. controls (15,20) and (15,10) .. (20,10) .. controls (25,10) and (25,0) .. (30,0);
    \draw[line width=6,color=white] (0,0) -- (20,0) .. controls (25,0) and (25,10) .. (30,10);
    \draw[line width=2] (0,0) -- (20,0) .. controls (25,0) and (25,10) .. (30,10);
    \draw[line width=6,color=white] (0,10) -- (10,10) ..controls (15,10) and (15,20) .. (20,20) -- (30,20);
    \draw[line width=2,color=white!35!black] (0,10) -- (10,10) ..controls (15,10) and (15,20) .. (20,20) -- (30,20);
    \draw[line width=6,color=white] (0,20) .. controls (5,20) and (5,30) .. (10,30) -- (30,30);
    \draw[line width=2,color=white!55!black] (0,20) .. controls (5,20) and (5,30) .. (10,30) -- (30,30);
  \end{scope}
 \end{tikzpicture}
 \end{equation}

 \begin{nota}
 For all $n\in\NN$ and $\kk\in[0,\nnp]$ we put
 \[
[\kk]_\nn=\begin{cases} 1& \text{if $\kk=\nnp$},\\
\nn & \text{if $\kk=0$},\\
\kk & \text{otherwise}.
\end{cases}
\]
Moreover for all integers $\ii$ and $\jj$ the symbol $\textbf{1}_{\ii=\jj}$ equals $1$ if the relation $\ii=\jj$ holds and $0$ otherwise.
 \end{nota}
 
As we can directly see on diagrams of~\eqref{E:delta}, for all $\kk\in[1,\nn]$ we have
 \begin{equation}
 \label{E:delta:pi}
  \pi(\delta_\nn)(\kk)=[k+1]_n \quad\text{and}\quad \pi(\delta_\nn\inv)(\kk)=[k-1]_n,
\end{equation}
moreover for all $1\leq \ii<\jj\leq\nn$ we have 
\begin{equation}
\label{E:delta:ell}
 \ell_{\ii,\jj}(\delta_\nn)=\textbf{1}_{\ii=1}\quad\text{and}\quad
\ell_{\ii,\jj}(\delta_\nn\inv)=-\textbf{1}_{\jj=\nn}.
\end{equation}

\begin{lemm}
 \label{L:phi}
 The automorphism $\overline{\phi}_\nn$ has order $\nn$ and for all $1\leq\indi<\indii\leq n$ we have 
 \[
\overline{\phi}_\nn(\aa\indi\indii)=\aa{[\indip]_n}{[\indiip]_n}
\]
with the convention $\aa\ii\jj=\aa\jj\ii$ whenever $\jj>\ii$ holds.
\end{lemm}

\begin{proof}
Computation of $\overline{\phi}_\nn(\aa\indi\indii)$ is an easy verification from Birman--Ko--Lee's presentation of $\BB\nn$.
The result on the order of $\overline{\phi}_\nn$ is then an immediate consequence.
\end{proof}

\begin{defi}
We denote by $\phi_\nn$ the homomorphism of $\ABKL\nn$-words defined for all integers $\indi$ and $\indii$ with $1\leq\indi<\indii\leq\nn$  by 
\[
\phi_\nn(\aa\indi\indii)=\aa{[\indip]_n}{[\indiip]_n}
\]
\end{defi}

By Lemma~\ref{L:phi}, for every $\ABKL\nn$-word $\uu$ we have 
\begin{equation}
\label{E:phi}
 \overline{\phi_\nn(\uu)}=\overline{\phi}_\nn(\overline\uu)=\delta_\nn\,\overline{u}\inv \, \delta_\nn\inv.
\end{equation}

\begin{prop}
\label{P:phi:stable}
The map $\phi_\nn$ is $\ABKL\nn$-stable. Moreover for every template~$t\in\TT\nn$ we have
\begin{align*}
\left(\phi_\nn^T(t)[\pi]\right)(\kk)=[1+t[\pi]([k-1]_n)]_n\quad \text{for all $k\in[1,\nn]$},\\
\phi_\nn^T(t)[\ell_{\ii,\jj}]=t[\ell_{[i-1]_n,[j-1]_n}]+\textbf{\emph{1}}_{\ii=1}-\textbf{\emph{1}}_{[1+t[\pi](n)]_n=\jj}\quad\text{for all $1\leq\ii<\jj\leq\nn$}.
\end{align*}
\end{prop}

\begin{proof}
The proof is similar to that of Proposition~\ref{P:PhiAcc}.
We detail only the case of Condition~$iii)$.
Let $\uu$ be a $\ABKL\nn$-word and $\kk$ be in $[1,\nn]$.
From \eqref{E:phi} and \eqref{E:delta:pi} we obtain
\begin{align*}
\pi(\overline{\phi_\nn(\uu)})(\kk)&=\pi(\delta_\nn)(\pi(\overline\uu)(\pi(\delta_\nn)\inv(\kk)))\\
&=\pi(\delta_\nn)(\pi(\overline\uu)([\kko]_\nn))\\
&=\left[1+\pi(\overline\uu)([\kko]_\nn)\right]_\nn.
\end{align*}
Let  $1\leq\ii<\jj\leq\nn$ be two integers. Lemma~\ref{L:inv} implies
\begin{align*}
\ell_{\ii,\jj}(\overline{\phi_\nn(\uu)})&=\ell_{\ii,\jj}(\delta_\nn\cdot\overline\uu\cdot\delta_\nn\inv)=\ell_{\ii,\jj}(\delta_\nn\cdot\overline\uu)+\ell_{\pi(\delta_\nn\cdot\overline\uu)\inv(\ii),\pi(\delta_\nn\cdot\overline\uu)\inv(\jj)}(\delta_\nn\inv).
\end{align*}
From~\eqref{E:delta:ell} we get that $\ell_{\pi(\delta_\nn\cdot\overline\uu)\inv(\ii),\pi(\delta_\nn\cdot\overline\uu)\inv(\jj)}(\delta_\nn\inv)$ is non zero iff $\pi(\delta_\nn\cdot\overline\uu)\inv(\jj)=\nn$, \ie, iff $\pi(\delta_\nn\cdot\overline\uu)(\nn)=\jj$ which is equivalent to $[1+\pi(\overline\uu)(\nn)]_\nn=\jj$.
We then obtain
\begin{align*}
\ell_{\ii,\jj}(\overline{\phi_\nn(\uu)})&=\ell_{\ii,\jj}(\delta_\nn\cdot\overline\uu)+\begin{cases}-1 & \text{if $[1+\pi(\overline\uu)(\nn)]_\nn=\jj$}, \\ 0 & \text{otherwise},\end{cases}\\
&=\ell_{\ii,\jj}(\delta_\nn\cdot\overline\uu)-\textbf{1}_{[1+\pi(\overline\uu)(\nn)]_\nn=\jj}.
\end{align*}
Moreover, by \eqref{E:delta:ell} we have 
\begin{align*}
\ell_{\ii,\jj}(\delta_\nn\cdot\overline\uu)&=\ell_{\ii,\jj}(\delta_\nn)+\ell_{\pi(\delta_\nn)\inv(\ii),\pi(\delta_\nn)\inv(\jj)}(\overline{u})\\
&=\textbf{1}_{\ii=1}+\ell_{[i-1]_\nn,[j-1]_\nn}(\overline{\uu}),
\end{align*}
with the convention  $\ell_{\indi,\indii}=\ell_{\indii,\indi}$ for $\indi>\indii$.
Eventually we obtain
\[
\ell_{\ii,\jj}(\overline{\phi_\nn(\uu)})=\ell_{[i-1]_\nn,[j-1]_\nn}(\overline{\uu})+\textbf{1}_{\ii=1}-\textbf{1}_{[1+\pi(\overline\uu)(\nn)]_\nn=\jj}.\qedhere
\]
\end{proof}

\subsection{Action on templates} 

We now describe an action of a subgroup of bijections of $\TT\nn$ on $\TT\nn$ itself.
Eventually, for any template $t\in\TT\nn$, braids of $\BB\nn(\SSS\nn,\ell,t)$ shall be in bijection with $\BB\nn(\SSS\nn,\ell,t')$ whenever $t'$ belongs in the orbit of $t$.

\begin{defi}
We define $G_{\ABP\nn}$, \resp $G_{\ABKL\nn}$, to be the subgroup of bijections of~$\TT\nn$ generated by $\{\mathrm{inv}_{\ABP\nn}^T,\theta_{\ABP\nn}^T,\Phi_\nn^T\}$, \resp by $\{\mathrm{inv}_{\ABKL\nn}^T,\phi_\nn^T\}$. 
For $t\in\TT\nn$, we denote
\[
G_{\ABP\nn}\star t=\{g(t),\  g\in G_{\ABP\nn}\}\quad\text{and}\quad G_{\ABKL\nn}\star t=\{g(t),\ g\in G_{\ABKL\nn}\}
\]
the orbits of $t$ under the action of $G_{\ABP\nn}$, \resp $G_{\ABKL\nn}$.
\end{defi}

\begin{lemm}
For $n\geq 3$, we have $G_{\ABP\nn}\simeq \left(\ZZ/2\ZZ\right)^3$ and $G_{\ABKL\nn}\simeq \ZZ/2\ZZ\times \ZZ/n\ZZ$.
\end{lemm}

\begin{proof}
Let $\uu=\sig{i_1}^{e_1}\cdots \sig{i_m}^{e_m}$ be a $\ABP\nn$-word with $e_1,\ldots,e_m\in\{-1,+1\}$.
From 
\begin{align*}
\mathrm{inv}_{\ABP\nn}(\theta_{\ABP\nn}(\uu))&=\mathrm{inv}_{\ABP\nn}(\sig{i_1}^{-e_1}\cdots \sig{i_m}^{-e_m})=\sig{i_m}^{e_m}\cdots\sig{i_1}^{e_1},\\
\theta_{\ABP\nn}(\mathrm{inv}_{\ABP\nn}(\uu))&=\theta_{\ABP\nn}(\sig{i_m}^{-e_m}\cdots\sig{i_1}^{-e_1})=\sig{i_m}^{e_m}\cdots\sig{i_1}^{e_1},
\end{align*}
we obtain that the maps $\mathrm{inv}_{\ABP\nn}$ and $\theta_{\ABP\nn}$ commute on $\ABP\nn$-words.
Let $t$ be a template of $\TT\nn$ and $\vv$ be a $\ABP\nn$-word representing a braid of template $t$. 
We have 
\[
(\mathrm{inv}_{\ABP\nn}^T\circ \theta_{\ABP\nn}^T)(t)=\tau\left(\overline{(\mathrm{inv}_{\ABP\nn}\circ \theta_{\ABP\nn})(\vv)}\right)=\tau\left(\overline{(\theta_{\ABP\nn}\circ\mathrm{inv}_{\ABP\nn})(\vv)}\right)=
 (\theta_{\ABP\nn}^T\circ\mathrm{inv}_{\ABP\nn}^T)(t)
\]
and so $\mathrm{inv}_{\ABP\nn}^T$ and $\theta_{\ABP\nn}^T$ commute.
Similar arguments establish the commutation of~$\theta_{\ABP\nn}$ and  $\Phi_\nn$, $\mathrm{inv}_{\ABP\nn}$ and $\Phi_\nn$, $\mathrm{inv}_{\ABKL\nn}$ and $\phi_\nn$.
We then obtain that $G_{\ABP\nn}$ and $G_{\ABKL\nn}$ are quotient of 
\[
H_n= \left<\mathrm{inv}_{\ABP\nn}^T\right>\times\left<\theta_{\ABP\nn}^T\right>\times \left<\Phi_n^T\right>\quad\text{and}\quad H_n^\ast= \left<\mathrm{inv}_{\ABKL\nn}^T\right>\times\left<\phi_n^T\right>.
\]
respectively.
The maps $\theta_{\ABP\nn}^T$ and $\mathrm{inv}_{\ABP\nn}^T$ have order~$2$ since it is the case for $\theta_{\ABP\nn}$ and $\mathrm{inv}_{\ABP\nn}$ by construction.
From Lemma~\ref{L:Phi} and Lemma~\ref{L:phi} the map $\Phi_\nn$ and $\phi_\nn$ have order $2$ and $\nn$ respectively.
We then obtain the ismorphisms $H_n\simeq\left(\ZZ/2\ZZ\right)^3$ and~$H_n^\ast\simeq\ZZ/2\ZZ\times \ZZ/n\ZZ$.

Immediate computations establish that $G_{\ABP\nn}\star \tau(\sig1\sig2^2)$ and  $G_{\ABKL\nn}\star \tau(\sig1)$ have respectively $8$ and $2n$ elements and so we obtain
$$
G_{\ABP\nn}=H_n\simeq \left(\ZZ/2\ZZ\right)^3\quad\text{and}\quad G_{\ABKL\nn}=H_n^\ast\simeq \ZZ/2\ZZ\times \ZZ/n\ZZ.\qedhere
$$
\end{proof}
From a geometric point of view, the maps $\mathrm{inv}_{\ABP\nn}^T$, $\theta_{\ABP\nn}^T$ and $\Phi_n^T\circ\theta_{\ABP\nn}^T$ can be seen as reflections along the coordinate planes of the $3$-space and the result on~$G_{\ABP\nn}$ is immediate. If we consider the base points of the braid evenly placed in a circumference, the maps $\mathrm{inv}_{\ABKL\nn}^T$ and $\phi_n^T$ correspond respectively to a reflection along a plane and to a rotation of order $n$ along an axis orthogonal to this plane, establishing the result on $G_{\ABKL\nn}$.

\begin{rema}
Note that for $n=2$, the map $\Phi_n$ and $\phi_n$ are trivial and that $\mathrm{inv}_{\ABP\nn}$ and~$\theta_{\ABP\nn}$ are equals. Hence we obtain $G_{\ABP\nn}\simeq \ZZ/2\ZZ\simeq G_{\ABKL\nn}$
\end{rema}

\begin{exam}
\label{X:Orb}
We recall that the template of a braid $\br\in B_4$ is 
\begin{equation}
\tau(\br)=\left(\pi(\br),\ell_{1,2}(\br),\ell_{1,3}(\br),\ell_{2,3}(\br),\ell_{1,4}(\br),\ell_{2,4}(\br),\ell_{3,4}(\br)\right)
\end{equation}

The template of $\sig1\siginv2$ seen in $B_4$ is $t=((1\,2\,3),1,-1,0,0,0,0)$.
Using Propositions~\ref{P:inv:stable}, \ref{P:theta:stable} and \ref{P:phi:stable} we obtain
\begin{align*}
\mathrm{inv}_{\ABP4}^T(t)&=((1\,3\,2),0,-1,1,0,0,0),\\
\theta_{\ABP4}^T(t)&=((1\,2\,3),-1,1,0,0,0,0),\\
\Phi_4^T(t)&=((2\,4\,3),0,0,0,0,-1,1),\\
(\mathrm{inv}_{\ABP4}^T(t)\circ\theta_{\ABP4}^T)(t)&=((1\,3\,2),0,1,-1,0,0,0),\\
(\mathrm{inv}_{\ABP4}^T\circ\Phi_4^T)(t)&=((2\,3\,4),0,0,1,0,-1,0),\\
(\theta_{\ABP4}^T\circ\Phi_4^T)(t)&=((2\,4\,3),0,0,0,0,1,-1),\\
(\mathrm{inv}_{\ABP4}^T\circ\theta_{\ABP4}^T\circ\Phi_4^T)(t)&=((2\,3\,4),0,0,-1,0,1,0),
\end{align*}
and so the set $G\star t$ has exactly $8$ elements.
\end{exam}

\subsection{Template reduction}

Now we define a total ordering on $\TT\nn$. We start with permutations of $\Sym\nn$.

\begin{defi}
For $\sigma$ and $\sigma'$ two permutations of $\Sym\nn$ we write $\sigma<\sigma'$ whenever
\[
(\sigma(1),\ldots,\sigma(\nn))<_\textsc{CoLex}(\sigma'(1),\ldots,\sigma'(\nn)),
\]
\ie, whenever there exists $\kk\in[1,\nn]$ such that $\sigma(\nn)=\sigma'(\nn),\ldots,\sigma(\kkp)=\sigma'(\kkp)$ and~$\sigma(\kk)<\sigma'(\kk)$.
\end{defi}

For example, the ordering of permutations occurring in Example~\ref{X:Orb} is 
\begin{equation}
\label{E:X:ordP}
(2\,3\,4)<(2\,4\,3)<(1\,2\,3)<(1\,3\,2).
\end{equation}

\begin{defi}
For two templates $t=(\sigma,(\ell_{\ii,\jj})_{1\leq\ii<\jj\leq\nn})$ and $t'=(\sigma',(\ell'_{\ii,\jj})_{1\leq\ii<\jj\leq\nn})$ we write $t<t'$ whenever 
\[
(\sigma,\ell_{1,2},\ldots,\ell_{1,\nn},\ldots,\ell_{\nno,\nn})<_\textsc{Lex} (\sigma',\ell'_{1,2},\ldots,\ell'_{1,\nn},\ldots,\ell'_{\nno,\nn}),
\]
where we recall the integers $\ell_{i,j}(\br)$ are enumerated following a co-lexicographic ordering on their indices (see Definition~\ref{D:Template}).
For a template $t$ we denote by $\red_{\SSS\nn}(t)$  the minimal element of $G_{\SSS\nn}\!\star t$.
We say that a template $t\in\TT\nn$ is \emph{$\SSS\nn$-reduced} if $\red_{\SSS\nn}(t)=t$ holds.
For an integer $\ll\in\NN$ we denote by~$R_\nn(\SSS\nn,\ell)$ the set of reduced templates lying in~$\TT\nn(\SSS\nn,\ell)$.
 
\end{defi}

\begin{exam}
We reconsider template $t$ of Example~\ref{X:Orb}.
By \eqref{E:X:ordP} we obtain
\[
\red_{\ABP4}(t)=((2\,3\,4),0,0,-1,0,1,0).
\]
which is equal to $(\mathrm{inv}_{\ABP4}^T\circ\theta_{\ABP4}^T\circ\Phi_4^T)(t)$.
\end{exam}

\begin{prop}
\label{P:TempRed}
For $\mu$ a $\SSS\nn$-stable map of $\SSS\nn$-words, $\ell$ an integer $\geq 0$ and  $t$ a template of $\TT\nn$ we have 
\[
\overline{\mu}(\BB\nn(\SSS\nn,\ell,t))=\BB\nn(\SSS\nn,\ell,\mu^T(t))
\]
and $\card{\BB\nn(\SSS\nn,\ell,t)}=\card{\BB\nn(\SSS\nn,\ell,\mu^T(t))}$.
\end{prop}

\begin{proof}
A direct consequense of Definition~\ref{D:Stable} and Lemma~\ref{L:Stable}.
\end{proof}

\begin{coro}
\label{C:Sum}
Let $\ell$ be an integer. 
We have 
\begin{align*}
s(\BB\nn,\SSS\nn;\ll)&=\sum_{t\in \RR\nn(\SSS\nn,\ell)} \card{\BB\nn(\SSS\nn,\ell,t)}\cdot\card{G_{\SSS\nn}\!\star t}.
\end{align*}
\end{coro}

\begin{proof}
We have 
\[
\BB\nn(\SSS\nn,\ell)=\bigsqcup_{t\in\TT\nn(\SSS\nn,\ell)} \BB\nn(\SSS\nn,\ell,t)=\bigsqcup_{t_r\in\RR\nn(\SSS\nn,\ell)} \bigsqcup_{t\in G_{\SSS\nn}\!\star t_r} \BB\nn(\SSS\nn,\ell,t)
\]
Assume $t_r$ is a template of $\RR\nn(\SSS\nn,\ell)$ and $t$ lies in $G_{\SSS\nn}\!\star t_r$. 
Then there exists a $\SSS\nn$-stable bijection $\mu\in G_{\SSS\nn}$ satisfying $t=\mu^T(t_r)$.
It follows from Proposition~\ref{P:TempRed} that the set~$\BB\nn(\SSS\nn,\ell,t)$ has the same cardinality as $\BB\nn(\SSS\nn,\ell,t_r)$. So we obtain
\begin{align*}
s(\BB\nn,\SSS\nn;\ll)&=\card{\BB\nn(\SSS\nn,\ell)}=\sum_{t_r\in\RR\nn(\SSS\nn,\ll)} \sum_{t\in G(\SSS\nn)\star t_r} \card{\BB\nn(\SSS\nn,\ell,t)}\\
%&=\sum_{t_r\in\RR\nn} \sum_{t\in G_\nn\cdot t_r} \card{\BB\nn(\ABP\nn,\ell,t_r)}\\
&=\sum_{t_r\in\RR\nn(\SSS\nn,\ll)}\card{\BB\nn(\SSS\nn,\ell,t_r)}\cdot\card{G_{\SSS\nn}\!\star t_r}.\qedhere
\end{align*}
\end{proof}

\subsection{Algorithmic improvement}

We now give an improvement of the algorithms of Section~\ref{S:Template} using Corollary~\ref{C:Sum}. 
From Corollary~\ref{C:Sum} we know how to obtain $s(\BB\nn,\SSS\nn;\ell)$ from an enumeration of braids associated to a reduced template.
As in Section~\ref{S:Template} we assume we have a function $\textsc{LoadRed}(\nn,\ell,t)$ loading from the storage memory a representative set of~$\BB\nn(\SSS\nn,\ell,t)$ where $t$ is a reduced template.
We also assume we have a function $\textsc{SaveRed}(\WW,\nn,\ell,t)$ which saves a representative set~$\WW$ of~$\BB\nn(\SSS\nn,\ell,t)$ whenever $t$ is a reduced template.

Enumerating only braids with a reduced template reduces the requirements of storage space.
But there is a little difficulty.
The template $t_x$ used in the call of \textsc{Load} line~6 of Algorithm~\ref{A:TempRepSet} -- \textsc{TempRepSet} is not necessarily reduced.
However, thanks to Proposition~\ref{P:TempRed} we have 
\[
\BB\nn(\SSS\nn,\ell,t)=\overline{g}\inv(\BB\nn(\SSS\nn,\ell,\red_{\SSS\nn}(t))), 
\]
where $g^T(t)=\red_{\SSS\nn}(t)$.
Hence if $\WW_r$ is a representative set of $\BB\nn(\SSS\nn,\ell,\red_{\SSS\nn}(t))$ then $\WW=g\inv(\WW_r)$ is a representative set of $\BB\nn(\SSS\nn,\ell,t)$.
We then obtain Algorithm~\ref{A:Rep2} -- \textsc{LoadFromRed} that can return any representative set of $\BB\nn(\SSS\nn,\ell,t)$ from the storage of braids of $\SSS\nn$-length $\ell$ with a reduced template.

\begin{algorithm}
\caption{-- \small \textsc{LoadFromRed} : Returns a representative set \texttt{W} of $\BB\nn(\SSS\nn,\ell,t)$ from the storage of representative sets of braids of length $\ell$ having a reduced template.}
\label{A:Rep2}
\small
\begin{algorithmic}[1]
\Function{LoadFromRed}{$\nn,\ell,t$}
 \State $t_r\gets \red_{\SSS\nn}(t)$  \Comment{the minimal element of $G_{\SSS\nn}\star t$}
\State determine $g \in G_{\SSS\nn}$ such that $t=g^T(t_r)$
\State $\texttt{\WW}_\texttt{r}\gets \textsc{LoadRed}(\nn,\ell,t_r)$
\State $\texttt{\WW}\gets \emptyset$
\For{$\texttt{\ww}\in \texttt{W}_\texttt{r}$}
    \State $\texttt{\WW}\gets \texttt{\WW}\sqcup\{g\inv(\texttt{\ww})\}$
\EndFor
\State{\textbf{return} $\texttt{\WW}$}
\EndFunction
\end{algorithmic}
\end{algorithm}

Replacing calls of \textsc{Load} by \textsc{LoadRed} and call of \textsc{Save} by \textsc{SaveRed} in Algorithm~\ref{A:TempRepSet}~--~\textsc{TempRepSet} we obtain Algorithm $\textsc{RedTempRepSet}(\ell,t)$ which saves a representative set $\texttt{\WW}_{\ll,t}$ of $\BB\nn(\SSS\nn,\ell,t)$ and returns the pair $(\card{\texttt{\WW}_\ell},\sum_{\uu\in\texttt{\WW}_\ell} \omega_{\SSS\nn}(\overline\uu))$ for every integer $\ell\geq 1$ and every reduced template $t$ of $\RR\nn(\SSS\nn,\ell)$.

By Corollary~\ref{C:Sum} the number $s(\BB\nn,\SSS\nn;\ell)$ can be determined by running Algorithm~\textsc{RedTempRepSet} on all reduced templates of $\TT\nn(\SSS\nn,\ell)$.
As for braids we can't determine reduced templates of $\TT\nn(\SSS\nn,\ell)$ considering only reduced templates of $\TT\nn(\SSS\nn,\llo)$.
Assume we dispose of the set $\RR\nn(\SSS\nn,\llo)$ of reduced templates of~$\TT\nn(\SSS\nn,\llo)$.
First we reconstruct the set $\TT\nn(\SSS\nn,\llo)$ using 
\[
 \TT\nn(\SSS\nn,\llo)=\{ g(t)\ \text{for} (g,t)\in G_{\SSS\nn}\times\RR\nn(\SSS\nn,\llo)\}.
\]
As a second step we use \eqref{E:TempSub} to obtain a supset $\TT\ell'$ of $\TT\nn(\SSS\nn,\ll)$.
Then we filter element of $\TT\ell'$ keeping only reduced templates by testing if a template is minimal in its orbit under the action of $G_{\SSS\nn}$.
Eventually we obtain the set $\RR\nn'(\SSS\nn,\ell)$ of reduced templates containing the reduced templates of $\TT\nn(\SSS\nn,\ell)$.
Moreover a template of~$\RR\nn'(\SSS\nn,\ell)$ is a reduced template of $\TT\nn(\SSS\nn,\ell)$ if and only if there exists a braid~$\br$ of $\BB\nn(\SSS\nn,\ell)$ having this precise template. We then obtain $\RR\nn(\SSS\nn,\ell)$ from  the set~$\RR\nn'(\SSS\nn,\ell)$.
These lead to Algorithm~\ref{A:NumB2}~--~\textsc{RedCombi}, which is an improved version of Algorithm~\ref{A:NumB}~--~\textsc{Combi}.

\begin{algorithm}
\caption{-- \small  \textsc{RedCombi} : Returns a pair of arrays of numbers $(n_\texttt{s},n_\texttt{g})$ satisfying $n_\texttt{s}[\ell]=s(\BB\nn,\SSS\nn;\ell)$ and $n_\texttt{g}[\ell]=g(\BB\nn,\SSS\nn;\ell)$ for all $\ell\leq\ell_{\text{max}}$ }
\label{A:NumB2}
\small
\begin{algorithmic}[1]
\Function{RedCombi}{$\ell_{\text{max}}$}
\State $n_\texttt{s}[0]\gets 1$;\ $n_\texttt{g}[0]\gets 1$
\State $R\gets \{(1_{\Sym\nn}, 0, \ldots, 0)\}$ \Comment{reduced templates of $\TT\nn(\SSS\nn,0)$}
\For{$\ell$ \textbf{from} $1$ \textbf{to} $\ell_{\text{max}}$}
	\State $R'\gets\emptyset$
	\State $n_\texttt{s}[\ell]\gets 0$;\ $n_\texttt{g}[\ell]\gets 0$
	\For{$t_r\in R$}
        \For{$t \in G_{\SSS\nn}\star t_r$}
            \For{$\texttt{x}\in \SSS\nn$}   
                \State $t'\gets t\ast \texttt{x}$
                \If{$t'$ is reduced and $t'\not\in R'$}
                    \State $(n'_\texttt{s},n'_\texttt{g})\gets \textsc{RedTempRepSet}(\ell,t')$
                    \If{$n'_\texttt{s}\not=0$}
                        \State $R'\gets R'\cup\{t'\}$
                        \State $n_\texttt{s}[\ell]\gets n_\texttt{s}[\ell]+n'_\texttt{s}\times\card{G_{\SSS\nn}\!\star t'}$
                        \State $n_\texttt{g}[\ell]\gets n_\texttt{g}[\ell]+n'_\texttt{g}\times\card{G_{\SSS\nn}\!\star t'}$
                    \EndIf
                \EndIf
            \EndFor 
		\EndFor
	\EndFor
	\State $R\gets R'$\Comment{reduced templates of $\TT\nn(\SSS\nn,\ell)$}
\EndFor
\State \textbf{return} $(n_\texttt{s},n_\texttt{g})$
\EndFunction
\end{algorithmic}
\end{algorithm}

\section{Results}
\label{S:Results}
For our experimentations we have coded a distributed version of Algorithm~\ref{A:NumB2}~--~\textsc{RedCombi} following a client / server model.
Roughly speaking the server runs the core of Algorithm~\ref{A:NumB2} while clients run Algorithm~\ref{A:Rep2}~--~\textsc{RedTempRepSet} in parallel. Technical details are voluntarily omitted. 
The source code of our program is available on GitHub \cite{github}.

These programs were executed on a single computational node\footnote{Financed by the project BQR CIMPA 2020 and the laboratory LMPA.} of the computing platform \texttt{CALCULCO}~\cite{Calculco}.
This node is equipped with 256~\texttt{Go} of RAM together with two processors AMD Epyc 7702 with 64 cores each for a total of~128 cores. 
In addition of this computational node we have used a distributed storage space of~30~\texttt{To} storing files containing representative sets.

\subsection{Three strands}

As values of $\mathcal{S}(\BB3,\ABP3)$ and $\mathcal{G}(\BB3,\ABP3)$ are already known since the work of L. Sabalka \cite{Sabalka} we have started our experimentation on the dual presentation of $\BB3$ (see Table~\ref{T:B3G}).

\begin{table}[h]

\footnotesize
 \begin{tabular}{r|r|r}
  $\ell$&$s(\BB3,\ABKL3;\ell)$& $g(\BB3,\ABKL3;\ell)$\\
  \hline
  0 & 1 & 1\\
  1 & 6 & 6\\
  2 & 20 & 30\\
  3 & 54 & 126\\
  4 & 134 & 498\\
  5 & 318 & 1\,926\\
  6 & 734 & 7\,410\\
  7 & 1\,662 & 28\,566\\
  8 & 3\,710 & 110\,658\\
  9 & 8\,190 & 431\,046\\
  10 & 17\,918 & 1\,687\,890\\
 \end{tabular}
 \qquad
 \begin{tabular}{r|r|r}
  $\ell$&$s(\BB3,\ABKL3;\ell)$& $g(\BB3,\ABKL3;\ell)$\\
  \hline
  11 & 38\,910 & 6\,639\,606 \\
  12 & 83\,966 & 26\,216\,418 \\
  13 & 180\,222 & 103\,827\,366\\
  14 & 385\,022 & 412\,169\,970\\
  15 & 819\,198 & 1\,639\,212\,246\\
  16 & 1\,736\,702 & 6\,528\,347\,778\\
  17 & 3\,670\,014 & 26\,027\,690\,886\\
  18 & 7\,733\,246 & 103\,853\,269\,650\\
  19 & 16\,252\,926 & 414\,639\,810\,486\\
  20 & 34\,078\,718 & 1\,656\,237\,864\,738\\
  21 & 71\,303\,166 & 6\,617\,984\,181\,606
 \end{tabular}
 \vspace{1em}
 \caption{\label{T:B3G} Combinatorics of $\BB3$ relatively to dual generators $\ABKL3$.}
\end{table}

Using Padé approximant on obtained values we can  conjecture rational expression for the spherical and geodesic growth series of $\BB3$ relatively to dual generators.

\begin{conj}
 The spherical and geodesic growth series of $\BB3$ relatively to dual generators are
\begin{equation*}
\mathcal{S}(\BB3,\ABKL3)=\frac{(t+1)(2t^2-1)}{(t-1)(2t-1)^2},
\qquad
\mathcal{G}(\BB3,\ABKL3)=\frac{12t^3-2t^2+3t-1}{(2t-1)(3t-1)(4t-1)}.
\end{equation*}
\end{conj}

If the previous conjecture is true the growth rate of $s(\BB3,\ABKL3;\ell)$ is $2$ while that of $g(\BB3,\ABKL3;\ell)$ is $4$.

\subsection{Four strands}

In her thesis \cite{Albenque-thesis}, M. Albenque computes the value $s(\BB4,\ABP4;\ell)$ up to $\ell\leq12$.
Running our algorithm on the $128$-cores node of the CALCULCO platform we determine the spherical and geodesic combinatorics of $\BB4$ relatively to Artin's generators up to length $25$ (see Table~\ref{T:B4A}).
\begin{table}[h]
\footnotesize
 \begin{tabular}{r|r|r}
  $\ell$&$s(\BB4,\ABP4;\ell)$&$g(\BB4,\ABP4;\ell)$\\
  \hline
  0&1&1\\
  1&6&6\\
  2&26&30\\
  3&98&142\\
  4&338&646\\
  5&1\,110&2\,870\\
  6&3\,542&12\,558\\
  7&11\,098&54\,026\\
  8&34\,362&229\,338\\
  9&105\,546&963\,570\\
  10&322\,400&4\,016\,674\\
  11&980\,904&16\,641\,454\\
  12&2\,975\,728&68\,614\,150
  \end{tabular}
  \qquad
   \begin{tabular}{r|r|r}
  $\ell$&$s(\BB4,\ABP4;\ell)$&$g(\BB4,\ABP4;\ell)$\\
  \hline
  13&9\,007\,466&281\,799\,158\\
  14&27\,218\,486&1\,153\,638\,466\\
  15&82\,133\,734&4\,710\,108\,514\\
  16&247\,557\,852&19\,186\,676\,438\\
  17&745\,421\,660&78\,004\,083\,510\\
  18&2\,242\,595\,598&316\,591\,341\,866\\
  19&6\,741\,618\,346&1\,283\,041\,428\,650\\
  20&20\,252\,254\,058&5\,193\,053\,664\,554\\
  21&60\,800\,088\,680&20\,994\,893\,965\,398\\
  22&182\,422\,321\,452&84\,795\,261\,908\,498\\
  23&547\,032\,036\,564&342\,173\,680\,884\,002\\
  24&1\,639\,548\,505\,920&1\,379\,691\,672\,165\,334\\
  25&4\,911\,638\,066\,620&5\,559\,241\,797\,216\,166\\
  \end{tabular}
   \vspace{1em} 
  
  \caption{\label{T:B4A} Combinatorics of $\BB4$ relatively to Artin's generators $\ABP4$.}
  \label{T:B4Artin}
  \end{table}
  Unfortunately the obtained values do not allow us to guess a rational expression of $\mathcal{S}(\BB4,\ABP4)$ or of $\mathcal{G}(\BB4,\ABP4)$.
  For information the storage of all braids of $\BB4$ with geodesic $\ABP4$-length $\leq 25$ and reduced templates requires $26$ \texttt{To} of disk space.
  
In case of dual generators we have reached length $17$ (see Table~\ref{T:B4D}).
\begin{table}[h]
\footnotesize
\begin{tabular}{r|r|r}
  $\ell$&$s(\BB4,\ABKL4;\ell)$&$g(\BB4,\ABKL4;\ell)$\\
  \hline
  0&1&1\\
  1&12&12\\
  2&84&132\\
  3&478&1\,340\\
  4&2\,500&12\,788\\
  5&12\,612&117\,452\\
  6&62\,570&1\,053\,604\\
  7&303\,356&9\,311\,420\\
  8&1\,506\,212&81\,488\,628
\end{tabular}
\qquad
 \begin{tabular}{r|r|r}
  $\ell$&$S(\BB4,\ABKL4;\ell)$&$g(\BB4,\ABKL4;\ell)$\\
  \hline
  9&7\,348\,366&708\,368\,540\\
  10&35\,773\,324&6\,128\,211\,364\\
  11&173\,885\,572&52\,826\,999\,612\\
  12&844\,277\,874&454\,136\,092\,148\\
  13&4\,095\,929\,948&3\,895\,624\,824\,092\\
  14&19\,858\,981\,932&33\,359\,143\,410\,468\\
  15&96\,242\,356\,958&285\,259\,736\,104\,444\\
  16&466\,262\,144\,180&2\,436\,488\,694\,821\,748\\
  17&2\,258\,320\,991\,652&20\,790\,986\,096\,580\,060
  \end{tabular}
  \vspace{1em}
  \caption{\label{T:B4D} Combinatorics of $\BB4$ relatively to dual generators $\ABKL4$.}
\end{table}
Using Padé approximant on our values we can conjecture the value of the spherical growth series of $\BB4$ relatively to dual generators.

\begin{conj}
The spherical growth series of $\BB4$ relatively to dual generators is
\begin{equation}
\label{E:SB4}
\mathcal{S}(\BB4,\ABKL4)=-\frac{(t+1)(10t^6-10t^5-3t^4+11t^3-4t^2-3t+1)}{(t-1)(5t^2-5t+1)(10t^4-20t^3+19t^2-8t+1)}
\end{equation}
\end{conj}

If the previous conjecture is true, the growth rate of $s(\BB4,\ABKL4;\ell)$ is given by the inverse of the maximal root of the denominator of \eqref{E:SB4}, which is approximatively~$4.8$.
 Unfortunately we are not able to formulate such a conjecture for the geodesic growth series of $\BB4$ relatively to dual generators.

 \vspace{1em}
\noindent \textbf{Acknowledgments.} The author wishes to thank the anonymous referee for his/her very sharp comments.

\bibliographystyle{plain}
\bibliography{biblio}

\end{document}